\documentclass[11pt,reqno]{amsart}
\usepackage{amssymb,amsmath,epsfig,mathrsfs, enumerate, xparse, mathtools}
\usepackage[pagewise]{lineno}
\usepackage{graphicx}\usepackage[normalem]{ulem}
\usepackage{fancyhdr}
\allowdisplaybreaks

\usepackage[margin=2.5cm]{geometry}
\usepackage{hyperref}
\hypersetup{
 colorlinks   = true,
 urlcolor     = black,
 linkcolor    = blue,
 citecolor   = red ,
 bookmarksopen=true
}
 \usepackage{tikz}
 \newcommand{\mb}{\mathbb}
 \newcommand{\mc}{\mathcal}
 
 \newcommand{\f}{\frac}
 \newcommand{\ld}{\lambda}

 \newcommand{\dd}{\partial}

 \newcommand{\nn}{\nonumber}

 \newcommand{\iy}{\infty}

 \newcommand{\tn}{\textnormal}
 \newcommand{\py}{  \partial_{x_{n+1}}^a} 
 \usepackage[notref,notcite]{}
 \usepackage[pagewise]{lineno}
 \newtheorem{theorem}{Theorem}[section]
 \newtheorem{lemma}[theorem]{Lemma}
 \newtheorem{proposition}[theorem]{Proposition}
 \newtheorem{corollary}[theorem]{Corollary}
 
 \newtheorem{remark}[theorem]{Remark}

 \numberwithin{equation}{section}

 \usepackage{hyperref}

 \newcommand{\R}{\rangle}

 \newcommand{\Rn}{\mathbb{R}^n}

 \newcommand{\bp}{\begin{prob}}
 	\newcommand{\bpr}{\begin{proof}}
 		\newcommand{\epr}{\end{proof}}

 	\newcommand{\ve}{\varepsilon}

 	\newcommand{\Rb}{\mathbb{R}}

		\usepackage{amsmath}
 	\theoremstyle{definition}

 	\title[vanishing order etc]{Sharp asymptotic of solutions to some nonlocal parabolic equations}

  \author{Agnid Banerjee}
\address{School of Mathematical and Statistical Sciences, Arizona State University, Tempe} \email[Agnid Banerjee]{agnidban@gmail.com}

\author{Abhishek Ghosh}
\address{Institute of Mathematics, Polish Academy of Sciences, Ul. \'Sniadeckich 8, 00-656 Warsaw, Poland.} \email[Abhishek Ghosh]{abhi170791@gmail.com; agosh@impan.pl}
\thanks{A.B is supported in part  by Department of Atomic Energy,  Government of India, under
project no.  12-R \& D-TFR-5.01-0520. A. G is supported by TIFR-CAM, Bangalore-560065, India.}

\keywords{}
\subjclass{35A02, 35B60, 35K05}

\begin{document}
		\begin{abstract}
		We show that if $u$ solves the fractional parabolic equation $(\partial_t - \Delta )^s u = Vu$ in $B_5 \times (-25, 0]$ ($0<s<1$) such that $u(\cdot, 0) \not\equiv 0$, then the maximal vanishing order of $u$ in space-time at $(0,0)$ is upper bounded by $C\left(1+\|V\|_{C^{1}_{(x,t)}}^{1/2s}\right)$. As $s \to 1$, it converges to the sharp maximal order of vanishing   due to Donnelly-Fefferman and Bakri.   This quantifies a space like strong unique continuation result recently proved in \cite{ABDG}.  The proof is achieved by means of a new quantitative Carleman estimate that we derive for the corresponding extension problem   combined with a quantitative monotonicity in time result  and a compactness argument.
		
	\end{abstract}
	
  \maketitle
  
  \tableofcontents

 	
 	%

  \section{Introduction}
In  $\Rn_x\times \mathbb R_t$ we consider the heat operator $H = \partial_t -\Delta_x$ and denote by $H^s$ its fractional power of order $s\in (0,1)$. In the recent work \cite[Theorem 1.1]{ABDG}, the following space like strong unique continuation result is proven.

\medskip

\noindent \textbf{Theorem A.}\label{abdg}
\emph{Let $u$ solve $H^s u=Vu$ in $B_1 \times (-1, 0]$ where $V \in C^{1}$ and $0<s<1$. If $u$ vanishes to infinite order in $(x,t)$ at $(0,0)$, i.e.}
\begin{equation}
||u||_{L^{\infty}(B_r \times (-r^2, 0])} = O(r^k),\ \text{for all $k \in \mathbb N$ as $r \to 0$,}\end{equation}

 \emph{then $u(\cdot, 0) \equiv 0$.}
 
 In order to provide the right context to the above result, we note that  an example of Frank Jones in \cite{Jr} shows that  in fact, there exists non-trivial solutions of the heat equation in $\mathbb R^{n+1}$ which are supported in a strip of the type $\mathbb R^{n} \times (t_1, t_2)$. If we restrict Jones' example to a finite cylinder, we infer that space-time propagation of zeros of infinite order fails for local solutions to parabolic equations. Therefore in view of such an example, only space-like propagation of zeros of infinite order claimed in Theorem A is the best possible, even for solutions to parabolic  PDE's (i.e. in the local case when $s \to 1$). In this connection, we mention that for local solutions to second order parabolic equations, space-like strong unique continuation results were proven in the remarkable works \cite{EF_2003, EFV_2006}. Theorem A above can thus be seen as a nonlocal counterpart of the ones in those papers.

 The purpose of this present work is to quantify the space-like strong unique continuation result stated  in Theorem A above, i.e. we address the following question:
 
 \medskip
 
 \noindent \textbf{Question A.} \emph{ Let $u$ solve $H^s u=Vu$ in $B_1 \times (-1, 0]$ where $V \in C^{1}$ and $0<s<1$.  Suppose $u(\cdot, 0) \not \equiv 0$. Then what is the maximal order of vanishing of $u$ at $(0,0)$?}

In this paper, we provide the following answer to this question. We denote $B_{r}$ to be the ball of radius $r$ centred at the origin in $\mathbb{R}^n,$ also let $\mathbb B_r^+= \{(z, z_{n+1})\in \mathbb{R}^n\times \mathbb{R}: |z|^2+|z_{n+1}|^2<r^2,  z_{n+1} >0\}$. For further notions and notations we refer to Section \ref{s:pr}. 

\subsection*{Statements of main result}
\begin{theorem}[Quantitative space-like strong unique continuation]\label{main}
Fix $0<s<1.$ Let $u \in \operatorname{Dom}(H^s)$ solve \begin{equation}\label{e0}
H^s u(x,t) = - V(x,t) u(x,t).
\end{equation} in $B_5 \times (-25, 0]$. Assume that $u(\cdot, 0)\not\equiv 0$ in $B_1.$ Then there exists $\tilde r= \tilde r(u)>0$ such that for all $r \leq \tilde r(u)$ one has
	\begin{align}\label{df}
		\int_{B_{r}\times (-r^2, 0]} u^2(x, t)dx\,dt \geq   r^{\mathcal{N}},
\end{align}
where $\mathcal{N}=M\left(\frac{1}{\int_{\mathbb B_1^+} U^2(X,0)x_{n+1}^adX}+ \operatorname{log}(M\Theta)+(\|V\|_1^{1/2s}+1)\right)$, $\Theta=\frac{\int_{\mathbb B_{5}^+\times [0, 25)} U^2( X, t)x_{n+1}^adXdt}{\int_{\mathbb B_{1}^+}U^2(X,0)x_{n+1}^adX},$ where $a:=1-2s,$ with $U$ being the solution to the backward extension problem   corresponding to $u$ as in \eqref{exprob} below and  
$$\|V\|_1\overset{def}{=}\|V\|_{L^{\infty}(B_5 \times (-25, 0])} +\|\nabla_x V\|_{L^{\infty}(B_5 \times (-25, 0])}+\|\partial_{t} V\|_{L^{\infty}(B_5 \times (-25, 0])}.$$ Here $M$ is some universal constant depending only on $n$ and $a$.
\end{theorem}

For a historical account,  we mention that for nonlocal elliptic equations of the type $(-\Delta)^s u = Vu$, a strong unique continuation result was obtained by Fall and Felli in \cite{FF}. Their analysis combined the frequency function approach in \cite{GL1, GL2} with the Caffarelli-Silvestre extension method in \cite{CS}. We also mention the interesting work of R\"uland \cite{Ru2} where instead the Carleman method is used together with \cite{CS} to obtain results similar to those in \cite{FF} but with weaker assumptions on the potential $V$. See also  \cite{Yu} where the case of nonlocal variable coefficient elliptic equations has been studied.  Finally, for global solutions of the nonlocal equation \eqref{e0}, a backward space-time strong unique continuation theorem was previously established by one of us with Garofalo in \cite{BG} ( see also the recent work \cite{FPS}) and more recently a space like strong unique continuation result for local solutions to \eqref{e0} as in Theorem A was proven in \cite{ABDG}. See also \cite{BS1} for a space-like strong unique continuation result for fractional parabolic Lam\'e type operators.  We also refer to \cite{AT} where the structure of the nodal set of solutions to \eqref{e0} has been studied. It is to be noted that in both the works \cite{ABDG, BG}, the approach is based on a monotonicity formula for an adjusted frequency function for the extension problem for $H^{s}$ ( which constitutes the non-local counterpart of the well known monotonicity forumula discovered by Poon in \cite{Po}) combined with an appropriate blowup argument with respect to the so-called Almgren type rescalings. Unique continuation for nonlocal equations have also found applications in the context of fractional inverse problems ( see for instance \cite{GSU,  LLR} and the references therein).

Now for a proper perspective on quantitative strong unique continuation results, we mention that in the papers \cite{DF1}, \cite{DF2}, Donnelly and Fefferman showed that if $u$ is an eigenfunction with eigenvalue  $\lambda$ on a  smooth, compact and connected $n$-dimensional Riemannian manifold $M$, then  the maximal vanishing order of $u$ is less than $C \sqrt{\lambda},$ where $C$ only depends on the manifold $M$. Using this estimate, they showed  that  if the Riemannian metric is real analytic, then  $H^{n-1}(\{x: u_{\lambda}(x)=0\})\leq C  \sqrt{\lambda},$ where $u_{\lambda}$ is the eigenfunction corresponding to $\lambda$ and therefore  gave a complete  answer to a famous  conjecture of Yau (\cite{Yau}) in the analytic setting. It is to be mentioned that in recent times, there has been some very interesting developments in the smooth setting as well, thanks to some breakthrough  works of Logunov and Malinnikova in \cite{L1, L2, LM}.   We note that the zero set of $u_{\lambda}$ is referred to as the nodal set. This order of vanishing is sharp. If, in fact, we consider $M = \mathbb S^n \subset \mathbb R^{n+1}$, and we take the spherical harmonic $Y_\kappa$ given by the restriction to $\mathbb S^n$ of the function $f(x_1,...,x_n,x_{n+1}) = \Re (x_1 + i x_2)^\kappa$, then one has $\Delta_{\mathbb S^n} Y_\kappa = - \lambda_\kappa Y_\kappa$, with $\lambda_\kappa = \kappa(\kappa+n-2)$, and the order of vanishing of $Y_\kappa$ at the North pole $(0,...,0,1)$ is precisely $\kappa = C \sqrt {\lambda_\kappa}$. 

In his work  \cite{Ku} ( see also \cite{Ku2}),  Kukavica considered the more general problem
\begin{equation}\label{e1}
\Delta u = V(x) u,
\end{equation}
where $V\in W^{1, \infty}$, and showed that the maximal vanishing order  of $u$ is bounded above by $C( 1+ \|V\|_{W^{1, \infty}})$. He also conjectured that the rate of vanishing order of $u$ is less than or equal to $C(1+ \|V\|_{L^{\infty}}^{1/2})$, which agrees with the Donnelly-Fefferman result when $V = - \lambda$. Employing Carleman estimates,  Bourgain and Kenig in \cite{BK} showed that the rate of  vanishing order  of $u$  is less than $C(1+ \|V\|_{L^{\infty}}^{2/3})$, and furthermore the exponent $\frac{2}{3}$ is sharp for complex  potentials $V$ based on a counterexample of Meshov (see \cite{Me}).

Not so long ago, the rate of vanishing order of $u$ has been shown to be less than $C(1+ \|V\|_{W^{1, \infty}}^{1/2})$  independently by Bakri in \cite{Bk} and Zhu in \cite{Zhu1}. Bakri's approach is based on  an extension of the Carleman method in \cite{DF1}.  On the other hand, Zhu's approach is based on a variant of the frequency function approach  employed by Garofalo and Lin in \cite{GL1}, \cite{GL2}, in the context of strong  unique  continuation problems. The approach of Zhu has been  subsequently extended in \cite{BG2}  to variable coefficient principal part with Lipschitz coefficients  where a similar quantitative  uniqueness result at the boundary of $C^{1, Dini}$ domains has been obtained. We would also like to mention that  in \cite{Zhu2}, an analogous  quantitative uniqueness result  has been established for solutions  to parabolic equations of the type
\begin{equation}\label{locpar}
\operatorname{div}(A(x,t) \nabla u) - u_t= Vu, \end{equation}
where $V \in C^{1}$  and $A(x,t) \in C^{2}$ by an adaption of an approach due to Vessella in \cite{Ve} ( see also \cite{EV}).  More precisely, in \cite{Zhu2} it is shown that if $u$ in a non-trivial solution to \eqref{locpar}   in $B_1 \times (-1,0)$, then one has 

\begin{equation}\label{nonvan}
\int_{B_r \times (-1, 0)} u^2 \geq C r^{C(1+ \|V\|_1^{1/2})},\ \ r\leq 1/2,\ C=C(u).
\end{equation}
It is to be mentioned that $C$ appearing in \eqref{nonvan} depends on some global quantities involving $u$. This  is optimal as examples in the plane such as  $u(z)= \Re(x_1+ix_2)^k$, show that even for harmonic functions, the vanishing order depends on global assumptions.  We also note that although \eqref{nonvan} generalizes the elliptic results, it doesn't quantify the local space like strong unique continuation result in \cite{EF_2003, EFV_2006} because \eqref{nonvan} constitutes a vanishing order estimate in space which is averaged over time.  Now  in a more recent work of one of us with Arya in \cite{AB1}, we have shown that if $u$ solves \eqref{locpar} in $B_1 \times (-1, 0]$ and $u(\cdot, 0) \not \equiv 0$, then the vanishing order of $u(\cdot, 0)$ can be upper bounded by $C(1+\|V\|_{C^1}^{1/2})$ which strictly refines the estimate of Zhu in \eqref{nonvan} above. The proof in \cite{AB1} is based on a novel quantitative version of the Carleman estimate in \cite{EF_2003, EFV_2006} which incorporates the $C^{1}$ norm of the  zero order perturbation $V$. Finally, we refer to the works \cite{CK} and \cite{KL} for  other variants of the  quantitative uniqueness results in the parabolic setting.

Now for nonlocal equations of the type 
\[
(-\Delta)^s u = Vu,\]
R\"uland in \cite{Ru1} showed that the vanishing order is proportional to $C_1\|V\|_{C^{1}}^{1/2s} + C_2$  which in the limit as $s \to 1$, exactly reproduces the result of Donnelly and Fefferman. This quantitative uniqueness estimate has also been applied to derive  a nonlocal  Landis type result in \cite{RW}, see also \cite{BG24} for a Landis-Oleinik-type result for the space-time fractional equation \eqref{e0} proving that exponential decay of super-Gaussian type is not possible at infinity. See also \cite{Zhu0} for vanishing order estimates for Steklov eigenfunctions which via the extension approach of Caffarelli and Silvestre in \cite{CS},   is related to the case $s=1/2$.  We also refer to  \cite{BL} for  earlier results on quantitative uniqueness for Steklov eigenvalue problems. 
\medskip

\noindent Now concerning fractional parabolic equations of the type \eqref{e0},  in a recent joint work of  one of us with Arya in \cite{AryaBan} we  have established the following nonlocal generalization of the estimate in \eqref{nonvan} for nontrivial solutions to \eqref{e0} in $B_1 \times (-1, 0)$
\begin{equation}\label{nonvan1}
\int_{B_r \times (-1, 0)} u^2 \geq C r^{C\left(1+ \|V\|_{C^1}^{1/2s}\right)},\ r\leq 1/2.
\end{equation}
The inequality \eqref{nonvan1} in \cite{AryaBan} is obtained  via a new quantitative Carleman estimate for the extension problem  \eqref{exprob} combined with a  boundary to bulk  propagation of smallness estimate.  Furthermore, since \eqref{nonvan1} in \cite{AryaBan} is derived by a direct  approach based on propagation of smallness estimate,  the constant $C=C(u)$ in \eqref{nonvan1} has an explicit dependence on the  function $U$  that solves \eqref{exprob} corresponding to $u$. On the other hand,  we observe that  in the estimate \eqref{df} in Theorem \ref{main}, $\tilde r(u)$ is implicit because an indirect compactness argument is crucially used in the proof. See the discussion in subsection \ref{key} below. However, since   \eqref{nonvan1} expresses  a vanishing order estimate in space which is averaged over time,  it doesn't quantify the space like strong unique continuation result Theorem A  in \cite{ABDG} which our main result Theorem \ref{main} does.

\subsection{Key ideas in the proof of Theorem \ref{main}:}\label{key}The key ingredients in the proof of Theorem \ref{main} are the following:

\begin{itemize}
	
	\item   A   generalization of the quantitative version of the Escauriaza-Fernandez-Vessella type Carleman estimate derived  in  \cite{AB1} to the setting of the fractional heat extension problem \eqref{exprob}.  See Theorem \ref{carl1} below. Such an estimate   has the precise quantitative dependence on the $C^{1}$ norm of the potential $V$  and  is the key novelty of our work. This entails a lot of new features not present in the local case dealt in \cite{AB1}. For instance, it turns out that  the class of weights $\sigma_s$  defined in Lemma \ref{sigma} that we use in the proof of  Theorem \ref{carl1} are tailored for the quantitative dependence on $\alpha$ that we obtain  in \eqref{al}. This is one of the major observations in this present work.  Over here, we also mention that a careful examination of the proof of the space like strong unique continuation result in \cite{ABDG} gives a  vanishing order  estimate which instead has an  exponential dependence  on the $C^{1}$ norm of $V$.

	\item The second key element in the proof of Theorem \ref{main} is the  quantitative monotonicity in time result in Lemma \ref{mont} below.  Lemma \ref{mont} in particular provides a sharp  ``quantitative passage"   of vanishing order information to $t=0$  of  the solution $U$ which solves \eqref{exprob}.  Now since our present setting involves a zero order perturbation of  a weighted Dirchlet to Neumann map as in \eqref{exprob},  the proof of such a result is significantly more  involved than its local counterpart in \cite{AB1} and is based on some delicate application of trace inequalities.

	\item After Theorem  \ref{carl1} and  Lemma \ref{mont} are established, we show that  a careful book keeping of the arguments in \cite{ABDG} leads to a quantitative space like  doubling inequality for solutions to \eqref{exprob} as in Theorem \ref{db2}. It is to be noted that in order to derive a space time doubling inequality from Theorem \ref{db2},  starting from Theorem \ref{db2},   if we follow the arguments  in \cite{EFV_2006} ( or its extension problem counterpart  in \cite{ABDG}), we would  get  an upper bound on the vanishing order proportional to   $(\|V\|_{1}^{1/2s}+1) \log(\|V\|_1+1)$ instead of $(\|V\|_1^{1/2s} +1)$. The difference in the nature of the ``space like'' and ``space time'' doubling constants is already seen in \cite[Theorem 3.5]{ABDG}. Therefore in order to circumvent this technical obstruction,  we instead argue by an indirect compactness argument to obtain a sharp  space time vanishing order estimate  for the solution $U$ to  \eqref{exprob}  as in Theorem \ref{main-doubling-cylinder} from Theorem \ref{db2}.	 This is another new feature of this work. Then by a further blowup argument which relates the solution $u$ of \eqref{e0} to the solution $U$ of  \eqref{exprob}, we derive Theorem \ref{main}. 
	\end{itemize}
	
	The paper is organized as follows. In section \ref{s:pr}, we introduce some basic notations and notions and gather some preliminary results that are relevant to this present work.  In section \ref{s:ce}, we prove our main Carleman estimate Theorem \ref{carl1} for solutions to \eqref{exprob}. In section \ref{s:quant}, we first prove a quantitative monotonicity in time result for solutions to \eqref{exprob} which then combined with the Carleman estimate in Theorem \ref{carl1} and  compactness  type arguments,  allows us to derive the sharp vanishing order estimate  in Theorem \ref{main-doubling-cylinder}. In section \ref{s:proof}, we finally prove our main result from Theorem \ref{main-doubling-cylinder}.

	\section{Preliminaries}\label{s:pr}
	
	In this section we introduce the relevant notation and gather some auxiliary  results that will be useful in the rest of the paper. Generic points in $\mathbb R^n \times \mathbb R$ will be denoted by $(x_0, t_0), (x,t)$, etc. For an open set $\Omega\subset \mathbb R^n_x\times \mathbb R_t$ we indicate with $C_0^{\infty}(\Omega)$ the set of compactly supported smooth functions in $\Omega$.  We also indicate by $H^{\alpha}(\Omega)$ the non-isotropic parabolic H\"older space with exponent $\alpha$ defined in \cite[p. 46]{Li}. The symbol $\mathscr S(\mathbb R^{n+1})$ will denote the Schwartz space of rapidly decreasing functions in $\mathbb R^{n+1}$. For $f\in \mathscr S(\mathbb R^{n+1})$ we denote its Fourier transform by 
\[
 \hat f(\xi,\sigma) = \int_{\mathbb R^n\times \mathbb R} e^{-2\pi i(\langle \xi,x\rangle + \sigma t)} f(x,t) dx dt = \mathscr F_{x\to\xi}(\mathscr F_{t\to\sigma} f).
\]
The heat operator in $\mathbb R^{n+1} = \mathbb R^n_x \times \mathbb R_t$ will be denoted by $H = \partial_t - \Delta_x$. Given a number $s\in (0,1)$ the notation $H^s$ will indicate the fractional power of $H$ that in  \cite[formula (2.1)]{Samko} was defined on a function $f\in \mathscr S(\mathbb R^{n+1})$ by the formula
\begin{equation}\label{sHft}
\widehat{H^s f}(\xi,\sigma) = (4\pi^2 |\xi|^2 + 2\pi i \sigma)^s\  \hat f(\xi,\sigma),
\end{equation}
with the understanding that we have chosen the principal branch of the complex function $z\to z^s$. We then introduce the natural domain for the operator $H^s$.
\begin{align}\label{dom}
\mathscr H^{2s} & =  \operatorname{Dom}(H^s)   = \{f\in \mathscr S'(\mathbb R^{n+1})\mid f, H^s f \in L^2(\mathbb R^{n+1})\}
\\
&  = \{f\in L^2(\mathbb R^{n+1})\mid (\xi,\sigma) \to (4\pi^2 |\xi|^2 + 2\pi i \sigma)^s  \hat f(\xi,\sigma)\in L^2(\mathbb R^{n+1})\},
\notag
\end{align} 
where the second equality is justified by \eqref{sHft} and Plancherel theorem. 
 It is important to keep in mind that definition \eqref{sHft} is equivalent to the one based on Balakrishnan formula (see \cite[(9.63) on p. 285]{Samko})
\begin{equation}\label{balah}
H^s f(x,t) = - \frac{s}{\Gamma(1-s)} \int_0^\infty \frac{1}{\tau^{1+s}} \big(P^H_\tau f(x,t) - f(x,t)\big) d\tau,
\end{equation}
where we have denoted by
\begin{equation}\label{evolutivesemi}
P^H_\tau f(x,t) = \int_{\Rn} G(x-y,\tau) f(y,t-\tau) dy = G(\cdot,\tau) \star f(\cdot,t-\tau)(x)
\end{equation}
the \emph{evolutive semigroup}, see \cite[(9.58) on p. 284]{Samko}. We refer to Section 3 in \cite{BG} for relevant details.

Henceforth, given a point $(x,t)\in \mathbb R^{n+1}$ we will consider the thick half-space $\mathbb R^{n+1} \times \mathbb R^+_{x_{n+1}}$. At times it will be convenient to combine the additional variable $x_{n+1}>0$ with $x\in \Rn$ and denote the generic point in the thick space $\Rn_x\times\mathbb R^+_{x_{n+1}}$ with the letter $X=(x,x_{n+1})$. For $x_0\in \mathbb R^n$ and $r>0$ we let $B_r(x_0) = \{x\in \Rn\mid |x-x_0|<r\}$,
$\mathbb B_r(X)=\{Z = (z,z_{n+1}) \in \mathbb R^n \times \mathbb R \mid |x-z|^2 + |x_{n+1}- z_{n+1}|^2 < r^2\}$. We also let $\mathbb B_r^+(X)= \mathbb B_r(X) \cap \{(z, z_{n+1}: z_{n+1} >0\}$. 
When the center $x_0$ of $B_r(x_0)$ is not explicitly indicated, then we are taking $x_0 = 0$. Similar agreement for the thick half-balls $\mathbb B_r^+(x_0,0)$. We will also use the $\mathbb Q_{r}$ for the set $\mathbb B_r \times [t_0,t_0+r^2)$ and $Q_r$ for the set  $ B_r \times [t_0,t_0+r^2).$ Likewise we denote  $\mathbb Q_r^+=\mathbb Q_r \cap \{(x,x_{n+1}): x_{n+1} > 0\}$. 
For notational ease $\nabla U$ and  $\operatorname{div} U$ will respectively refer to the quantities  $\nabla_X U$ and $ \operatorname{div}_X U$.  The partial derivative in $t$ will be denoted by $\partial_t U$ and also at times  by $U_t$. The partial derivative $\partial_{x_i} U$  will be denoted by $U_i$. At times,  the partial derivative $\partial_{x_{n+1}} U$  will be denoted by $U_{n+1}$.

 We next introduce the extension problem associated with $H^s$.  
 Given a number $a\in (-1,1)$ and a $u:\mathbb R^n_x\times \mathbb R_t\to \mathbb R$ we seek a function $U:\mathbb R^n_x\times\mathbb R_t\times \mathbb R_{x_{n+1}}^+\to \mathbb R$ that satisfies the boundary-value problem
\begin{equation}\label{la}
\begin{cases}
\mathscr{L}_a U \overset{def}{=} \partial_t (x_{n+1}^a U) - \operatorname{div} (x_{n+1}^a \nabla U) = 0,
\\
U((x,t),0) = u(x,t),\ \ \ \ \ \ \ \ \ \ \ (x,t)\in \mathbb R^{n+1}.
\end{cases}
\end{equation}
The most basic property of the Dirichlet problem \eqref{la} is that if $s = \frac{1-a}2\in (0,1)$ and $u \in \text{Dom}(H^{s})$, then we have the following convergence  in $L^{2}(\mathbb R^{n+1})$
\begin{equation}\label{np}
2^{-a}\frac{\Gamma(\frac{1-a}{2})}{\Gamma(\frac{1+a}{2})} \py U((x,t),0)=  - H^s u(x,t),
\end{equation}
where $\py$ denotes the weighted normal derivative
\begin{equation}\label{nder}
\py U((x,t),0)\overset{def}{=}   \lim\limits_{x_{n+1} \to 0^+}  x_{n+1}^a \partial_{x_{n+1}} U((x,t),x_{n+1}).
\end{equation}

When $a = 0$ ($s = 1/2$) the problem \eqref{la} was first introduced in \cite{Jr1} by Frank Jones, who in such case also constructed the relevant Poisson kernel and proved \eqref{np}. More recently Nystr\"om and Sande in \cite{NS} and Stinga and Torrea in \cite{ST} have independently extended the results in \cite{Jr1} to all $a\in (-1,1)$. 	

With this being said, we now suppose that $u$ be a solution to \eqref{e0}  and consider the weak solution $U$ of the following version of \eqref{la} (for the precise notion of weak solution of \eqref{wk} we refer to \cite[Section 4]{BG}) 
\begin{equation}\label{wk}
\begin{cases}
\mathscr{L}_a U=0 \ \ \ \ \ \ \ \ \ \ \ \ \ \ \ \ \ \ \ \ \ \ \ \ \ \ \ \ \ \ \text{in}\ \mathbb R^{n+1}\times \mathbb R^+_{x_{n+1}},
\\
U((x,t),0)= u(x,t)\ \ \ \ \ \ \ \ \ \ \ \ \ \ \ \ \text{for}\ (x,t)\in \mathbb R^{n+1},
\\
\py U((x,t),0)=  2^{a} \frac{\Gamma(\frac{1+a}{2})}{\Gamma(\frac{1-a}{2})} V(x,t) u(x,t)\ \ \ \ \text{for}\ (x,t)\in B_5 \times (-25,0].
\end{cases}
\end{equation}
 Note that the third equation in \eqref{wk} is justified by \eqref{e0} and \eqref{np}.   Further, in \cite[Lemma 5.3]{BG} the following regularity result for such weak solutions was proved. Such result will be relevant to our analysis.

\begin{lemma}\label{reg1}
Let $U$  be a weak solution of \eqref{wk} where $V  \in C^{1}$. Then there exists $\alpha>0$ such that one has up to the thin set $\{x_{n+1}=0\}$ 
\[
U_i,\ U_t,\ x_{n+1}^a U_{x_{n+1}}\ \in\ H^{\alpha}(\mathbb B_4^+ \times (-16, 0]),\ \ \ \  i=1,2,..,n.
\]
Moreover, the relevant H\"older norms are bounded by $\int_{\mathbb B_5^+ \times (-25, 0]} U^2 x_{n+1}^a dX dt$. 
\end{lemma}

We also need the following  weak unique continuation result from \cite{LLR} which is needed to analyse the blowup limit in the proof of Theorem \ref{main}. We refer the reader to the proof of Proposition 5.6 in \cite{LLR} where such a result has been established. See also \cite[Corollary 1.2]{BG1} where the same result has been derived as a consequence of the space like analyticity of  solutions to  \[\operatorname{div}(|x_{n+1}|^a \nabla U_0) -|x_{n+1}|^a \partial_t U_0=0\] that are symmetric in $x_{n+1}$ variable across $\{x_{n+1}=0\}$.
\begin{proposition}\label{wucp}
Let $U_0$ be a  weak solution to
\begin{equation}\label{homog}
\begin{cases}
\operatorname{div}(x_{n+1}^a \nabla U_0) -x_{n+1}^a \partial_t U_0=0\ \text{ in  $\mathbb B_1^+ \times (-1, 0]$,}
\\
\py U_0((x,t),0) \equiv 0\ \text{for all $(x,t)  \in B_1 \times (-1, 0]$,}
\end{cases}
\end{equation}
such that $U_0((x,t),0) \equiv 0$ for all $(x,t) \in B_1 \times (-1, 0]$. 
Then $U_0 \equiv 0$ in $\mathbb B_1^+ \times (-1, 0]$.
\end{proposition}

For notational purposes it will be convenient to work with the following backward version of problem \eqref{wk}.	
	
	\begin{equation}\label{exprob}
		\begin{cases}
			x_{n+1}^a \partial_t U + \operatorname{div}(x_{n+1}^a \nabla U)=0\ \ \ \ \ \ \ \ \ \ \ \ \text{in} \ \mb{B}_5^+ \times [0, 25),
			\\	
			U((x,t), 0)= u(x,t)
			\\
			\py U((x,t), 0)= Vu\ \ \ \ \ \ \ \ \ \text{in}\ B_5 \times [0,25).
		\end{cases}
	\end{equation}
	
	We note that the former can be transformed into the latter by changing $t \to -t$. 
	
The corresponding extended backward parabolic operator will be denoted as
	\begin{align}\label{extop}
		\widetilde{\mc{H}}_s & := x_{n+1}^a \dd_t + \tn{div} \left( x_{n+1}^a \nabla \right).
	\end{align}	
	
	We now state a few key auxiliary results that will be needed in the proof of our main Carleman estimate in Theorem \ref{carl1}.  The following lemma below provides us with the appropriate Carleman weight function that is tailored for deriving  our main Carleman estimate \eqref{care1}  with the quantitative dependence on $\alpha$ as in \eqref{al}, where the large parameter $\alpha$ is related to the order of vanishing. As previously mentioned in the introduction, this is one of the key observations in this work.

\begin{lemma}[\cite{EF_2003}] \label{sigma}
	Let $s\in (0, 1).$ Define
		\begin{equation}\label{theta'}\theta_{s}(t) = t^{s} \left( \log \f{1}{t} \right)^{1+s}.\end{equation}  Then the solution to the ordinary differential equation 
		$$\frac{d}{dt} \log \left(\frac{\sigma_{s}}{t\sigma_{s}'}\right)= \frac{\theta_{s}(\lambda t)}{t},~\sigma_{s}(0)=0,~\sigma_{s}'(0)=1,$$
		where $\lambda >0,$ has the following properties when $0\leq \lambda t\leq 1/2$:
		\begin{enumerate}
			\item $t e^{-N} \leq \sigma_{s}(t) \leq t,$
			\item $e^{-N} \leq \sigma_{s}'(t)\leq 1,$
			\item $|\partial_t[\sigma_{s} \log \frac{\sigma_{s}}{\sigma_{s}' t}]|+|\partial_t[\sigma_{s} \log \frac{\sigma_{s}}{\sigma_{s}' }]|\leq 3N$,
			\item $\left|\sigma_{s} \partial_t \left(\frac{1}{\sigma_{s}'}\partial_t[\log \frac{\sigma_{s}}{\sigma_{s}'(t)t}]\right)\right| \leq 3N e^{N} \frac{\theta_{s}(\gamma t)}{t},$
		\end{enumerate}	
		where $N$ is some universal constant. 
	\end{lemma}
\begin{proof}
We need to verify the following three conditions in order to apply the Lemma~4 in \cite{EF_2003}:
\begin{enumerate}[i)]
\item
Since $\theta_{s}(t)\to 0,$ as $t\to 0$ we clearly have $0\leq \theta_{s}(t)\leq N$ for some universal $N,$ possibly depending on $s.$
\item
Observe that
\begin{align}
\bigg|\frac{t \theta'(t)}{\theta(t)}\bigg|\leq \f{st^s \left(\log \f{1}{t}\right)^{1+s}+t^s (1+s)\left(\log \f{1}{t}\right)^{s}}{t^s \left(\log \f{1}{t}\right)^{1+s}}\leq N,
\end{align}
for $0< \lambda t\leq 1/2.$
\item
Performing a change of variables and using $s>0$ we obtain, $$\int_{0}^{1}\left(1+\log \f{1}{t}\right)\frac{\theta_{s}(t)}{t}=\int_{0}^{\infty}(1+t)e^{-ts} t^{1+s}\, dt\leq N<\infty.$$
\end{enumerate}
Now we apply Lemma~4 in \cite{EF_2003} to conclude our result.
\end{proof}
	
	 We also need the following Hardy type inequality in the Gaussian space which  can be found in Lemma 2.2 in \cite{ABDG}.  This can be regarded as the weighted analogue of  Lemma 3 in \cite{EFV_2006}.
	\begin{lemma}[Hardy type inequality]\label{hardy}
		For all $h \in C_0^{\infty}(\overline{\mb{R}^{n+1}_+})$ and $b>0$ the following inequality holds
		\begin{align*}
			& \int_{\mb{R}^{n+1}_+} x_{n+1}^a h^2 \frac{|X|^2}{8b} e^{-|X|^2/4b}dX \leq  2b \int_{\mb{R}^{n+1}_+} x_{n+1}^a|\nabla h|^2 e^{-|X|^2/4b} dX
			\\
			& + \frac{n+1+a}{2} \int_{\mb{R}^{n+1}_+}x_{n+1}^a h^2  e^{-|X|^2/4b}  dX.
		\end{align*}
	\end{lemma}
The next lemma is crucially needed in the proof of the doubling inequality for \eqref{exprob}.
The proof is completely analogous to that of \cite[Lemma 4]{EFV_2006} for the case $a=0$, so we omit it.

	\begin{lemma}\label{do}
		Assume that $N \ge 1,$ $h \in C_{0}^{\infty}(\overline{\mathbb R^{n+1}_+})$ and the inequality \begin{align*}
			2b \int_{\mathbb R^{n+1}_+} x_{n+1}^a|\nabla h|^2 e^{-|X|^2/4b}dX + \frac{n+1+a}{2}\int_{\mathbb{R}^{n+1}_+}x_{n+1}^a h^2 e^{-|X|^2/4b}dX \le N  \int_{\mathbb{R}^{n+1}_+} x_{n+1}^ah^2  e^{-|X|^2/4b}dX
		\end{align*}
		holds for $a \le \frac{1}{12N }.$ Then
		\begin{align}
			\int_{\mathbb B_{2r}^+}h^2 x_{n+1}^a dX \le e^N  \int_{\mathbb B_{r}^+}h^2 x_{n+1}^a dX 
		\end{align}
		when $0 < r \le 1/2.$
	\end{lemma}

We use the following time-independent trace inequality. For its proof see  \cite[p. 65]{Ru}.

\begin{lemma}[Trace inequality]\label{tr}
Let   $f\in C_0^\infty(\overline{\mathbb R^{n+1}_+})$. There exists a constant $C_0 = C_0(n,a)>0$ such that for every $A>1$ one has
\[
\int_{\Rn} f(x,0)^2 dx \le C_0 \left(A^{1+a} \int_{\mathbb R^{n+1}_+} f(X)^2 x_{n+1}^a dX + A^{a-1} \int_{\mathbb  R^{n+1}_+} |\nabla f(X)|^2 x_{n+1}^a dX\right).
\]
\end{lemma}

\section{The main carleman estimate}\label{s:ce}
We now state and prove our main Carleman estimate which constitutes the generalization of the Carleman estimate  in \cite[Lemma 6]{EFV_2006}  to degenerate operators of the type \eqref{extop}.  In the case when $V \equiv 0$, such an estimate has been established in \cite{ BVS, LLR}. The key new  feature of the estimate in \eqref{care1}  is the quantitative  dependence of the weight parameter $\alpha$ on the $C^{1}$ norm of $V$. See \eqref{al} below.  In order to prove such an estimate,  we adapt  some ideas from the recent work \cite{AB1}. However as previously  mentioned in the introduction, such an adaptation  is  very delicate because in our present situation,  the zero order $C^{1}$ perturbation   is associated with the corresponding weighted Dirchlet to Neumann map as in \eqref{extop}. From now on, we will let

\begin{equation}\label{V1}
\|V\|_1= \|V\|_{C^1_{(x,t)}(Q_5)} +1.
\end{equation}
Furthermore, in the rest of the discussion, we will let 
\begin{equation}\label{a}
a=1-2s.
\end{equation}

\begin{theorem}\label{carl1}
Let $ \widetilde{\mc{H}}_s$ be the   backward in time extension operator  in \eqref{extop}. Let $ w \in C_0^\iy \left( \overline{\mb{B}^+_4} \times [0,\left. \f{1}{e\ld}\right) \right)$  such that $\py w\equiv Vw$ on $\{x_{n+1}=0\}$ where $ \ld = \frac{\alpha}{ \delta^2}$ for some  $ \delta \in (0,1).$ Further, assume 
\begin{equation}\label{al}
\alpha\geq M(1+\|V\|_{1}^{\f{1}{2s}}),
\end{equation} where $M$ is a large universal constant and $\delta$ sufficiently small. Then 
		\begin{align}\label{care1}
			& \alpha^2 \int_{\mb{R}^{n+1}_+ \times [c, \iy)} x_{n+1}^{a} \sigma_{s}^{-2 \alpha}(t) \ w^2 \ G  + \alpha \int_{\mb{R}^{n+1}_+ \times [c, \iy)} x_{n+1}^{a} \sigma_{s}^{1-2 \alpha}(t)\  |\nabla w|^2 \ G  \\
			& \leq M  \int_{\mb{R}^{n+1}_+ \times [c, \iy)} \sigma_{s}^{ 1-2 \alpha}(t) x_{n+1}^{-a} \ \lvert \widetilde{\mc{H}_{s}} w \rvert ^2 \ G\notag\\
   &+ \sigma_{s}^{-2 \alpha}(c) \left\{ -\f{c}{M} \int_{t=c} x_{n+1}^a \ |\nabla w(X,c)|^2 \ G(X,c) \ dX + M\alpha \int_{t=c} x_{n+1}^a \ |w(X,c)|^2 \ G(X,c) \ dX\right\}.\notag
		\end{align}
		Here $\sigma_{s}$ is as in Lemma \ref{sigma},  $G(X,t) = \f{1}{{t^\f{n+1+a}{2}}} e^{-\f{|X|^2}{4t}}$ and $0 < c \le \f{1}{5\ld}. $
	\end{theorem}
	\begin{proof}
 		Let $\theta_{s}$ be as in Lemma \ref{sigma}.   For $t\in[0,\f{1}{e\ld})$, we first make the preliminary observation that
		\begin{align}\label{large}
			\f{\theta_{s}(\ld t)}{t} \ge \ld^{s} t^{s-1} \left( \log e \right)^{1+s} \gtrsim \lambda^s (e\lambda)^{1-s} \left( \log e \right)^{1+s} \gtrsim \ld .
		\end{align}
Moreover, as mentioned in the introduction, the choice of the function $\theta_{s}$ plays a pervasive role in the proof. The solid integrals below will be taken in $ \mb{R}^n \times [c, \iy) $ where $ 0 < c \le \f{1}{\ld} $ and we refrain from mentioning explicit limits in the rest of our discussion. Note that 
		\[ x_{n+1}^{-\f{a}{2}} \widetilde{\mc{H}}_s = x_{n+1}^{\f{a}{2}} \left( \dd_t + \tn{div}(\nabla) + \f{a}{x_{n + 1}} \dd_{n+1}\right).\]
		Define 
		\[ w(X,t) = \sigma_{s}^\alpha(t) e^{\f{|X|^2}{8t}} v(X,t). \]
		We then compute
		\begin{align}\label{sto1}
			\dd_{t} w = e^{\f{|X|^2}{8t}} \left( \sigma_{s}^\alpha(t) \dd_{t}v + \alpha \sigma_{s}^{\alpha - 1}(t) \sigma_{s}'(t) v - \f{|X|^2}{8t^2}\sigma_{s}^\alpha(t) v \right), \ \nabla w = e^{\f{|X|^2}{8t}} \sigma_{s}^\alpha(t) \left( \nabla v + \f{X}{4t} v \right).
		\end{align}
Therefore,
\begin{align*}
\tn{div} (\nabla w)& = \tn{div} \left( \sigma_{s}^\alpha(t) e^{\f{|X|^2}{8t}}  \left( \nabla v + \f{X}{4t} v \right) \right) \\ 
& = \sigma_{s}^\alpha(t) e^{\f{|X|^2}{8t}} \left[ \tn{div} (\nabla v) + \f{\langle X,  \nabla  v \rangle} {2t} + \left( \f{|X|^2}{16 t^2} + \f{n+1} {4t} \right) v \right].
\end{align*}
		Now we define the vector field 
		\begin{equation}\label{defz} \mc{Z} := 2t \dd_t + X \cdot \nabla. \end{equation}
		
		Note that $\mc{Z}$ is the infinitesimal generator of the parabolic dilations $\{\delta_r\}$ defined by $\delta_r(X,t)=(rX, r^2 t)$. In terms of $\mc{Z}$ we have

		\begin{align*}
			x_{n+1}^{-\f{a}{2}} \sigma_{s}^{-\alpha}(t) e^{-\f{|X|^2}{8t}} \widetilde{\mc{H}}_sw 
			& = x_{n+1}^{\f{a}{2}}  \left[ \tn{div}\left( \nabla v\right) + \f{1}{2t} \mc{Z}v + \left( \f{n+1 + a} {4t} + \f{\alpha \sigma_{s}'} {\sigma_{s}}\right) v - \f{|X|^2}{16t^2} v  + \f{a}{x_{n+1}} \dd_{n+1} v \right].
		\end{align*}

		Next we consider the expression
		\begin{align}\label{ex1}
			&  \int \sigma_{s}^{-2 \alpha}(t) t^{-\mu} x_{n+1}^{-a} e^{-\f{|X|^2}{4t}} \left(\f{t \sigma_{s}'}{\sigma_{s}}\right)^{-\f{1}{2}}  \lvert \widetilde{\mc{H}_s} w \rvert ^2\notag \\ 
			&= \int x_{n+1}^{a} t^{-\mu} \left(\f{t \sigma_{s}'}{\sigma_{s}}\right)^{-\f{1}{2}} \left[ \tn{div}\left( \nabla v\right) + \f{1}{2t} \mc{Z}v + \left( \f{n+1 + a} {4t} + \f{\alpha \sigma_{s}'} {\sigma_{s}}\right) v - \f{|X|^2}{16t^2} v  + \f{a}{x_{n+1}} \dd_{n+1} v \right]^2,
		\end{align}
		where $\mu$ will be appropriately chosen later. Then we estimate the integral \eqref{ex1} from below with an application of the algebraic inequality
		\[ \int P^2 + 2 \int PQ \le \int \left( P + Q \right)^2,\]
		where $P$ and $Q$ are chosen as
		\begin{align*}
			& P = \f{x_{n+1}^{\f{a}{2}} t^{-\f{\mu + 2}{2}}}{ 2} \left(\f{t \sigma_{s}'}{\sigma_{s}}\right)^{-\f{1}{4}} \mc{Z}v, \\ 
			& Q = x_{n+1}^{\f{a}{2}} t^{-\f{\mu}{2}} \left(\f{t \sigma_{s}'}{\sigma_{s}}\right)^{-\f{1}{4}}\left[ \tn{div}\left( \nabla v\right)  + \left( \f{n+1 + a} {4t} + \f{\alpha \sigma_{s}'} {\sigma_{s}}\right) v - \f{|X|^2}{16t^2} v  + \f{a}{x_{n+1}} \dd_{n+1} v \right].
		\end{align*} 
		We compute the terms coming from the cross product, i.e. from $\int PQ.$ We write $$ \int PQ: = \sum_{k=1}^4 \mc{I}_k,$$ where
		\begin{align*}
			&  \mc{I}_1 = \int x_{n+1}^{a} t^{-\mu} \left(\f{t \sigma_{s}'}{\sigma_{s}}\right)^{-\f{1}{2}} \f{1}{2t} \mc{Z}v \left( \f{n+1 + a} {4t} + \f{\alpha \sigma_{s}'} {\sigma_{s}} \right)v, \\
			&  \mc{I}_2 = \int x_{n+1}^a t^{-\mu} \left(\f{t \sigma_{s}'}{\sigma_{s}}\right)^{-\f{1}{2}} \f{\mc{Z}v} {2t} \ \tn{div} \left( \nabla v \right), \\
			& \mc{I}_3 = \int x_{n+1}^a t^{-\mu} \left(\f{t \sigma_{s}'}{\sigma_{s}}\right)^{-\f{1}{2}} \f{\mc{Z}v} {2t} \left(- \f{|X|^2}{16 t^2}\right) v, \\
			&  \mc{I}_4 = \int x_{n+1}^{a} t^{-\mu} \left(\f{t \sigma_{s}'}{\sigma_{s}}\right)^{-\f{1}{2}} \f{\mc{Z}v} {2t} \f{a\, \dd_{n+1}v}{x_{n+1}}.
		\end{align*}
		
\noindent{\textbf{Estimate for $\mc{I}_1:$}}
		\begin{align}
			\nn \mc{I}_1 & = \int x_{n+1}^{a} t^{-\mu} \left(\f{t \sigma_{s}'}{\sigma_{s}}\right)^{-\f{1}{2}} \f{1}{2t} \mc{Z}v \left( \f{n+1 + a} {4t} + \f{\alpha \sigma_{s}'} {\sigma_{s}} \right)v\label{e2}.
		\end{align}

We estimate the first term. By integrating by parts in $X$ and $t$ we have
\begin{align}
&  \f{n+1+a}{8} \int x_{n+1}^{a} t^{-\mu-2} \left(\f{t \sigma_{s}'}{\sigma_{s}}\right)^{-\f{1}{2}} \mc{Z}\left( \f{v^2}{2} \right)=\f{n+1+a}{8} \int x_{n+1}^{a} t^{-\mu-2} \left(\f{t \sigma_{s}'}{\sigma_{s}}\right)^{-\f{1}{2}}\bigg(t \partial_{t}(v^2)+\left \langle \frac{X}{2}, \nabla (v^2) \right \rangle\bigg) \\
   &=\f{n+1+a}{8} \int x_{n+1}^{a} t^{-\mu-2}(\mu+1) \left(\f{t \sigma_{s}'}{\sigma_{s}}\right)^{-\f{1}{2}} v^2+\f{(n+1+a)}{16} \int x_{n+1}^{a} t^{-\mu-1} \left(\f{t \sigma_{s}'}{\sigma_{s}}\right)^{-\f{3}{2}}\left(\f{t \sigma_{s}'}{\sigma_{s}}\right)' v^2 \notag  \\
			&  - \left( \f{n+1+a}{8} \right) \ c^{-\mu-1} \left(\f{c \sigma_{s}'(c)}{\sigma_{s}(c)}\right)^{-\f{1}{2}} \int_{t=c} x_{n+1}^{a} v^2(X,c) \ dX\notag\\
   &-\left( \f{n+1+a}{8} \right)\int \f{n+1+a}{2} x_{n+1}^{a} t^{-\mu-2} \left(\f{t \sigma_{s}'}{\sigma_{s}}\right)^{-\f{1}{2}} v^2, \label{ektaineq}
		\end{align}
where in the last line we used that $\operatorname{div}(Xx_{n+1}^a)=(n+1+a)x_{n+1}^a.$ If we  now let \begin{equation}\label{mu}
			\mu=\frac{n-1+a}{2}\end{equation} 
in \eqref{ektaineq}, then the first and fourth term on the right hand side cancel each other. Moreover, for this choice of $\mu,$ we find using integration by parts
\begin{align}\label{e40}
			& \f{\alpha}{2} \int x_{n+1}^{a} t^{-\mu-2} \left(\f{t \sigma_{s}'}{\sigma_{s}}\right)^{\f{1}{2}} \mc{Z}\left( \f{v^2}{2} \right) \\
			& =- \f{\alpha}{4} \int \tn{div}(x_{n+1}^a t^{-\f{n+3+a}{2}} \mc{Z})\left(\f{t \sigma_{s}'}{\sigma_{s}}\right)^{\f{1}{2}} v^2  -\f{\alpha}{4} \int x_{n+1}^{a} t^{-\mu-1} \left(\f{t \sigma_{s}'}{\sigma_{s}}\right)^{-\f{1}{2}} \left(\f{t \sigma_{s}'}{\sigma_{s}}\right)' v^2\notag\\&  - \f{\alpha}{2} c^{-\mu-1} \left(\f{c \sigma_{s}'(c)}{\sigma_{s}(c)}\right)^{-\f{1}{2}} \int_{t=c} x_{n+1}^{a} v^2(X,c) \ dX\notag \\
		&=-\f{\alpha}{4} \int x_{n+1}^{a} t^{-\mu-1} \left(\f{t \sigma_{s}'}{\sigma_{s}}\right)^{-\f{1}{2}} \left(\f{t \sigma_{s}'}{\sigma_{s}}\right)' v^2 - \f{\alpha}{2} c^{-\mu-1} \left(\f{c \sigma_{s}'(c)}{\sigma_{s}(c)}\right)^{-\f{1}{2}} \int_{t=c} x_{n+1}^{a} v^2(X,c) \ dX. \notag	
		\end{align}
		Here we used that $\tn{div}(x_{n+1}^a t^{-\f{n+3+a}{2}} \mc{Z})=0$. Therefore, for large enough $\alpha$ we obtain for some universal $N>1$
\begin{align}
 \mc{I}_{1}&:=\f{(n+1+a)}{16} \int x_{n+1}^{a} t^{-\mu-1} \left(\f{t \sigma_{s}'}{\sigma_{s}}\right)^{-\f{3}{2}}\left(\f{t \sigma_{s}'}{\sigma_{s}}\right)' v^2 - \left( \f{n+1+a}{8} \right) \ c^{-\mu-1} \left(\f{c \sigma_{s}'(c)}{\sigma_{s}(c)}\right)^{-\f{1}{2}} \int_{t=c} x_{n+1}^{a} v^2(X,c) \notag\\
&-\f{\alpha}{4} \int x_{n+1}^{a} t^{-\mu-1} \left(\f{t \sigma_{s}'}{\sigma_{s}}\right)^{-\f{1}{2}} \left(\f{t \sigma_{s}'}{\sigma_{s}}\right)' v^2  - \f{\alpha}{2} c^{-\mu-1} \left(\f{c \sigma_{s}'(c)}{\sigma_{s}(c)}\right)^{-\f{1}{2}} \int_{t=c} x_{n+1}^{a} v^2(X,c) \ dX\notag\\
&\geq \f{\alpha}{N} \int x_{n+1}^{a} t^{-\mu-1} \f{\theta_{s}(\lambda t)}{t} v^2- \alpha c^{-\mu-1} \left(\f{c \sigma_{s}'(c)}{\sigma_{s}(c)}\right)^{-\f{1}{2}}\int_{t=c} x_{n+1}^{a} v^2(X,c) \ dX. \label{mci1} \end{align} Notice that the fact $-\left(\f{t \sigma_{s}'}{\sigma_{s}}\right)'$ is a comparable to the quantity $\f{\theta_{s}(\lambda t)}{t}$ which follows from Lemma \ref{sigma} is being used in the last inequality. \\

\medskip
			
\noindent{\textbf{Estimate for $\mc{I}_2:$}} Now we consider the term $\mc{I}_2$ which  finally provides the positive  gradient terms in our Carleman estimate.  This is obtained via a Rellich type argument. We have		\begin{align}\label{k0}
\mc{I}_2 &= \int x_{n+1}^a t^{-\mu} \left(\f{t \sigma_{s}'}{\sigma_{s}}\right)^{-\f{1}{2}} \f{\mc{Z}v} {2t} \ \tn{div} \left( \nabla v \right)\notag\\
   &=\int x_{n+1}^a t^{-\mu} \left(\f{t \sigma_{s}'}{\sigma_{s}}\right)^{-\f{1}{2}} \left(\partial_{t} v+\f{X. \nabla v}{2t} \right) \ \tn{div} \left( \nabla v \right)=:\mc{I}_{21}+\mc{I}_{22}.
  \end{align}
We estimate them individually. Using divergence theorem, we have\begin{align}
\mc{I}_{21}&=-\int x_{n+1}^a t^{-\mu} \left(\f{t \sigma_{s}'}{\sigma_{s}}\right)^{-\f{1}{2}} v_{i} \dd_{t}(v_{i})-a \int x_{n+1}^{a-1} t^{-\mu} \left(\f{t \sigma_{s}'}{\sigma_{s}}\right)^{-\f{1}{2}} v_{n+1} \dd_t v\notag\\
&-\int_{\{x_{n+1}=0\}}t^{-\mu}\left(\f{t\sigma_{s}'}{\sigma_{s}}\right)^{-1/2}  V(x, t)v \partial_{t}v\ \text{(using $\py v= Vv$)}\notag \\
&=\f{1}{2}\int x_{n+1}^{a} (-\mu)t^{-\mu-1} \left(\f{t \sigma_{s}'}{\sigma_{s}}\right)^{-\f{1}{2}}|\nabla v|^2 \notag-\f{1}{4}\int x_{n+1}^{a} t^{-\mu} \left(\f{t \sigma_{s}'}{\sigma_{s}}\right)^{-\f{3}{2}}\left(\f{t \sigma_{s}'}{\sigma_{s}}\right)^{'} |\nabla v|^{2}\notag\\
&+\f{1}{2}\int_{\{t=c\}} x_{n+1}^{a} c^{-\mu} \left(\f{c \sigma_{s}'}{\sigma_{s}}\right)^{-\f{1}{2}}|\nabla v(X, c)|^{2}\notag\\
&\underbrace{-a\int x_{n+1}^{a} t^{-\mu} \left(\f{t \sigma_{s}'}{\sigma_{s}}\right)^{-\f{1}{2}} \f{\mc{Z}v} {2t} \f{ \dd_{n+1}v}{x_{n+1}}}_{-\mc{I}_{4}}+\f{a}{2}\int x_{n+1}^{a-1} t^{-\mu-1} \left(\f{t \sigma_{s}'}{\sigma_{s}}\right)^{-\f{1}{2}} (X, \nabla v)  \dd_{n+1}v\notag\\
& -\int_{\{x_{n+1}=0\}}t^{-\mu}\left(\f{t\sigma_{s}'}{\sigma_{s}}\right)^{-1/2}  V(x, t)v \partial_{t}v.\label{mci21}
\end{align}
We also have
\begin{align}
\mc{I}_{22}&=- \f{1}{2} \int t^{-\mu-1} \left(\f{t \sigma_{s}'}{\sigma_{s}}\right)^{-\f{1}{2}} \langle\nabla\left( x_{n+1}^a \langle X, \nabla v \rangle \right), \nabla(v) \rangle\notag\\
&-\f{1}{2} \int_{\{x_{n+1}=0\}} t^{-\mu-1} \left(\f{t \sigma_{s}'}{\sigma_{s}}\right)^{-\f{1}{2}}V(x, t)v \langle x, \nabla_{x}v\rangle\notag \\
& =-\f{a}{2}\int x_{n+1}^{a-1} t^{-\mu-1} \left(\f{t \sigma_{s}'}{\sigma_{s}}\right)^{-\f{1}{2}} (X, \nabla v)  \dd_{n+1}v\notag\\
&-\f{1}{2}\int t^{-\mu-1}\left(\f{t \sigma_{s}'}{\sigma_{s}}\right)^{-\f{1}{2}} x_{n+1}^{a}(X_{i}v_{ip}+v_{p})v_{p}\notag-\f{1}{2} \int_{\{x_{n+1}=0\}} t^{-\mu-1}\left(\f{t \sigma_{s}'}{\sigma_{s}}\right)^{-\f{1}{2}} Vv \langle x, \nabla_{x}v\rangle\notag\\
& =-\f{a}{2}\int x_{n+1}^{a-1} t^{-\mu-1} \left(\f{t \sigma_{s}'}{\sigma_{s}}\right)^{-\f{1}{2}} (X, \nabla v)  \dd_{n+1}v-\f{1}{2}\int t^{-\mu-1}\left(\f{t \sigma_{s}'}{\sigma_{s}}\right)^{-\f{1}{2}} x_{n+1}^{a}|\nabla v|^2\notag\\
&-\f{1}{4}\int x_{n+1}^{a} t^{-\mu-1}\left(\f{t \sigma_{s}'}{\sigma_{s}}\right)^{-\f{1}{2}}(X, \nabla(|\nabla v|^2))- \f{1}{2} \int_{\{x_{n+1}=0\}} t^{-\mu-1}\left(\f{t \sigma_{s}'}{\sigma_{s}}\right)^{-\f{1}{2}} V(x, t)v \langle x, \nabla_{x}v\rangle.\notag\end{align}
Now by integrating by parts the following term 

\[
-\f{1}{4}\int x_{n+1}^{a} t^{-\mu-1}\left(\f{t \sigma_{s}'}{\sigma_{s}}\right)^{-\f{1}{2}}(X, \nabla(|\nabla v|^2))
\]
in the above expression we obtain
\begin{align}
&\mc{I}_{22}=-\f{a}{2}\int x_{n+1}^{a-1} t^{-\mu-1} \left(\f{t \sigma_{s}'}{\sigma_{s}}\right)^{-\f{1}{2}} (X, \nabla v)  \dd_{n+1}v+\f{\mu}{2}\int t^{-\mu-1}\left(\f{t \sigma_{s}'}{\sigma_{s}}\right)^{-\f{1}{2}} x_{n+1}^{a}|\nabla v|^2\notag\\
&- \f{1}{2} \int_{\{x_{n+1}=0\}} t^{-\mu-1}\left(\f{t \sigma_{s}'}{\sigma_{s}}\right)^{-\f{1}{2}} V(x, t)\,v \langle x, \nabla_{x}v\rangle.\label{mci22}
\end{align}
Combining \eqref{k0}, \eqref{mci21}, and \eqref{mci22} with $\mc{I}_4$ we have 		
\begin{align}
\mc{I}_2 + \mc{I}_4 &=  -\f{1}{4}\int x_{n+1}^{a} t^{-\mu} \left(\f{t \sigma_{s}'}{\sigma_{s}}\right)^{-\f{3}{2}}\left(\f{t \sigma_{s}'}{\sigma_{s}}\right)^{'} |\nabla v|^{2}\notag+\f{1}{2}\int_{\{t=c\}} x_{n+1}^{a} c^{-\mu} \left(\f{c \sigma_{s}'}{\sigma_{s}}\right)^{-\f{1}{2}}|\nabla v(X, c)|^{2}\notag\\
&-\f{1}{2} \int_{\{x_{n+1}=0\}} t^{-\mu-1}\left(\f{t \sigma_{s}'}{\sigma_{s}}\right)^{-\f{1}{2}} {V}(x, t)\, v\, \langle x, \nabla_{x}v\rangle\notag\\
&-\int_{\{x_{n+1}=0\}}t^{-\mu}\left(\f{t\sigma_{s}'}{\sigma_{s}}\right)^{-1/2}  V(x, t) \partial_{t}\left(\f{v^2}{2}\right).\label{notreq}
\end{align}
Recall that
\begin{equation}\label{uv} \nabla v = \sigma_{s}^{-\alpha}(t) e^{-\f{|X|^2}{8t}} \left( \nabla w - \f{X}{4t} w \right). \end{equation}
Let us now consider the term  $- \f{1}{4} \int x_{n+1}^a t^{-\mu} \left(\f{t \sigma_{s}'}{\sigma_{s}}\right)^{-\f{3}{2}} \left( \f{t \sigma_{s}'}{\sigma_{s}} \right)' |\nabla v|^2$. 	Using \eqref{uv} we obtain
\begin{align}\label{i1}
& - \f{1}{4} \int x_{n+1}^a t^{-\mu} \left(\f{t \sigma_{s}'}{\sigma_{s}}\right)^{-\f{3}{2}} \left( \f{t \sigma_{s}'}{\sigma_{s}} \right)' \langle \nabla v,  \nabla v \rangle \\
&=-\f{1}{4} \int x_{n+1}^a t^{-\mu} \left(\f{t \sigma_{s}'}{\sigma_{s}}\right)^{-\f{3}{2}} \left( \f{t \sigma_{s}'}{\sigma_{s}} \right)' \sigma_{s}^{-2\alpha}(t) \left\langle \nabla w - \f{X}{4t} w, \left( \nabla w - \f{X}{4t} w \right) \right\rangle e^{-\f{|X|^2}{4t}}\notag \\
&=-\f{1}{4} \int x_{n+1}^a t^{-\mu} \left(\f{t \sigma_{s}'}{\sigma_{s}}\right)^{-\f{3}{2}} \left( \f{t \sigma_{s}'}{\sigma_{s}} \right)' \sigma_{s}^{-2\alpha}(t) \left( \langle \nabla w, \nabla w \rangle + \f{|X|^2}{16 t^2} w^2 - \f{1}{4t} \ \langle X \cdot \nabla (w^2) \rangle \right)  \ e^{-\f{|X|^2}{4t}}\notag  \\
& =-\f{1}{4} \int x_{n+1}^a t^{-\mu} \left(\f{t \sigma_{s}'}{\sigma_{s}}\right)^{-\f{3}{2}} \left( \f{t \sigma_{s}'}{\sigma_{s}} \right)' \sigma_{s}^{-2\alpha}(t) \left( \langle \nabla w, \nabla w \rangle - \f{|X|^2}{16t^2} w^2 \right) \ e^{-\f{|X|^2}{4t}}\notag \\
& - \f{1}{16} \int  t^{-\mu-1} \left(\f{t \sigma_{s}'}{\sigma_{s}}\right)^{-\f{3}{2}} \left(\f{t \sigma_{s}'}{\sigma_{s}} \right)' \tn{div}\left(\  x_{n+1}^a X \right) w^2e^{-\f{|X|^2}{4t}}\notag \\
& = - \f{1}{4} \int x_{n+1}^a t^{-\mu} \left(\f{t \sigma_{s}'}{\sigma_{s}}\right)^{-\f{3}{2}} \left( \f{t \sigma_{s}'}{\sigma_{s}} \right)' \sigma_{s}^{-2\alpha}(t) \left( |\nabla w|^2 - \frac{|X|^2}{16t^2} \right) \ e^{-\f{|X|^2}{4t}} \notag\\
&- \f{n+1+a}{16} \int x_{n+1}^a t^{-\mu-1} \left(\f{t \sigma_{s}'}{\sigma_{s}}\right)^{-\f{3}{2}} \left(\f{t \sigma_{s}'}{\sigma_{s}} \right)' \sigma_{s}^{-2\alpha}(t) w^2 e^{-\f{|X|^2}{4t}}.\notag 
\end{align}
The boundary integral in \eqref{notreq}  above,  i.e.  the term \[ \f{1}{2} c^{-\mu} \left(\f{c \sigma_{s}'(c)}{\sigma_{s}(c)}\right)^{-\f{1}{2}} \int_{t=c} x_{n+1}^{a} \langle \nabla v, \nabla v \rangle (X,c) \] can be computed in a similar fashion to obtain the following
\begin{align*}
& \f{1}{2} c^{-\mu} \left(\f{c \sigma_{s}'(c)}{\sigma_{s}(c)}\right)^{-\f{1}{2}} \int_{t=c} x_{n+1}^{a} \langle \nabla v, \nabla v \rangle (X,c) \ dX \\
			& = \f{1}{2} c^{-\mu} \sigma_{s}^{-2\alpha}(c) \left(\f{c \sigma_{s}'(c)}{\sigma_{s}(c)}\right)^{-\f{1}{2}} \int_{t=c} x_{n+1}^{a} \left( \langle \nabla w,  \nabla w \rangle - \f{|X|^2}{16c^2} w^2 + \f{n+1+a}{4c} w^2 \right) \ e^{-\f{|X|^2}{4c}} \ dX.
		\end{align*} 
\medskip

\noindent{\textbf{Estimate for $\mc I_3$:}} Let us now compute $\mc{I}_3$. We have
		\begin{align}\label{I_3}
			\mc{I}_3  & =-\f{1}{16} \int x_{n+1}^a t^{-\mu-2} \left(\f{t \sigma_{s}'}{\sigma_{s}}\right)^{-\f{1}{2}} \f{\mc{Z}v} {2t}  |X|^2 v \\
			& = -\f{1}{32} \int x_{n+1}^a t^{-\mu-2} \left(\f{t \sigma_{s}'}{\sigma_{s}}\right)^{-\f{1}{2}} |X|^2 \ \dd_t(v^2)\notag \\
			&  -\f{1}{64} \int x_{n+1}^a t^{-\mu-3} \left(\f{t \sigma_{s}'}{\sigma_{s}}\right)^{-\f{1}{2}} |X|^2 \ \langle X,  \nabla (v^2) \rangle\notag \\
			& = -\f{n+3+a}{64} \int x_{n+1}^a t^{-\mu-3} \left(\f{t \sigma_{s}'}{\sigma_{s}}\right)^{-\f{1}{2}} |X|^2 \ v^2\ \text{(using $\mu=\frac{n-1+a}{2}$)} \notag \\
			& -\f{1}{64} \int x_{n+1}^a t^{-\mu-2} \left(\f{t \sigma_{s}'}{\sigma_{s}}\right)^{-\f{3}{2}} \left(\f{t \sigma_{s}'}{\sigma_{s}}\right)' |X|^2 v^2 +\f{1}{32} c^{-\mu-2} \left( \f{c \sigma_{s}'(c)}{\sigma_{s}(c)} \right)^{-\f{1}{2}} \int_{t=c} x_{n+1}^a |X|^2 v^2\notag \\
			&  + \f{1}{64} \int t^{-\mu-3} \left(\f{t \sigma_{s}'}{\sigma_{s}}\right)^{-\f{1}{2}} |X|^2  (n+1+a) x_{n+1}^a v^2\notag  \\
			& + \f{1}{32} \int x_{n+1}^a t^{-\mu-3} \left(\f{t \sigma_{s}'}{\sigma_{s}}\right)^{-\f{1}{2}} |X|^2 v^2\notag \\
			& = -\f{1}{64} \int x_{n+1}^a t^{-\mu-2} \left(\f{t \sigma_{s}'}{\sigma_{s}}\right)^{-\f{3}{2}} \left(\f{t \sigma_{s}'}{\sigma_{s}}\right)' |X|^2 \sigma_{s}^{-2\alpha}(t)  w^2 e^{-\f{|X|^2}{4t}}\notag \\
   &+\f{1}{32} c^{-\mu-2} \left( \f{c \sigma_{s}'(c)}{\sigma_{s}(c)} \right)^{-\f{1}{2}} \int_{t=c} x_{n+1}^a |X|^2 \sigma_{s}^{-2\alpha}(t)  w^2 e^{-\f{|X|^2}{4t}} \label{mci3}.
		\end{align}	
Now we use the fact that  $- \left(\f{t\sigma_{s}'}{\sigma_{s}}\right)' \sim \f{\theta_{s}(\ld t)}{t}$ since the term $\f{t\sigma_{s}'}{\sigma_{s}}$ is positively bounded from both sides in view of Lemma \ref{sigma} and combining the above estimates we get for  a new universal $N$ that the following estimate holds
\begin{align}
\notag &\mc{I}_{1}+\mc{I}_{2}+ \mc{I}_{3}+ \mc{I}_{4} \\
&\geq \f{\alpha}{N} \int x_{n+1}^{a} \sigma_{s}^{-2\alpha}(t)  \f{\theta_{s}(\lambda t)}{t}\, G\, w^2+ \frac{1}{N}  \int x_{n+1}^a \f{\theta_{s}(\lambda t)}{t} \sigma_{s}^{1-2\alpha}(t)\,\, G\,\, |\nabla w|^2\notag\\
&-N\alpha \sigma_{s}^{-2\alpha}(c) \int_{t=c} x_{n+1}^{a} w^2(X,c)\,G(X, c)+\f{c}{N}\, \sigma_{s}^{-2\alpha}(c) \int_{t=c} x_{n+1}^{a} |\nabla w|^2 G \ dX\notag\\
&-\f{1}{2} \int_{\{x_{n+1}=0\}} t^{-\mu-1}\left(\f{t \sigma_{s}'}{\sigma_{s}}\right)^{-\f{1}{2}} {V}(x, t)\, v\, \langle x, \nabla_{x}v\rangle-\int_{\{x_{n+1}=0\}}t^{-\mu}\left(\f{t\sigma_{s}'}{\sigma_{s}}\right)^{-1/2}  V(x, t) \partial_{t}\left(\f{v^2}{2}\right).\label{combined-est}
\end{align}
At this point the proof depends on the precise estimates for the functions $\sigma_s$ and $\theta_{s}.$ \noindent Let us estimate the spatial boundary term in \eqref{combined-est}. Using the divergence formula we obtain the following alternate representation of such boundary terms.  

 \begin{align*}
 &K_{1}:=\f{1}{4} \int_{\{x_{n+1}=0\}} t^{-\mu-1}\left(\f{t \sigma_{s}'}{\sigma_{s}}\right)^{-\f{1}{2}} (V(x, t) n+\langle x, \nabla_{x} V(x, t)\rangle ) v^2,\\
 &K_{2}:=-\int_{\{x_{n+1}=0\}}t^{-\mu}\left(\f{t\sigma_{s}'}{\sigma_{s}}\right)^{-1/2}  V(x, t) \partial_{t}\left(\f{v^2}{2}\right).
 \end{align*}
 Using the trace inequality Lemma \ref{tr} we have 
\begin{align}
\notag|K_{1}| &\leq \int t^{-\mu-1} \sigma_{s}^{-2\alpha}\int_{\Rb^n}\left(nV(x, t)+\langle x, \nabla_{x}V(x, t) \R \right)e^{-\frac{|x|^2}{4t}} w^2\notag\\
 &\lesssim \|V\|_{1} \int t^{-\mu-1} \sigma_{s}^{-2\alpha} \int_{\Rb^n} e^{-\frac{|x|^2}{4t}} w^2 \label{k1mid}\\
\notag & \lesssim \|V\|_{1} \int t^{-\mu-1} \sigma_{s}^{-2\alpha} \left(A(t)^{1+a}\int_{\Rb^{n+1}_{+}} x_{n+1}^a e^{-\frac{|X|^2}{4t}} w^2+A(t)^{a-1}\int_{\Rb^{n+1}_{+}} x_{n+1}^a\big|\nabla w-w \f{X}{4t}\big|^2 e^{-\frac{|X|^2}{4t}} \right)\\
\notag & \lesssim \|V\|_{1} \int t^{-\mu-1} \sigma_{s}^{-2\alpha} \left(A(t)^{1+a}\int_{\Rb^{n+1}_{+}} x_{n+1}^a e^{-\frac{|X|^2}{4t}} w^2+A(t)^{a-1}\int_{\Rb^{n+1}_{+}} x_{n+1}^a |\nabla w|^2 e^{-\frac{|X|^2}{4t}}\right.\\
&\left.+A(t)^{a-1}\int_{\Rb^{n+1}_{+}} x_{n+1}^{a} w^2 \f{|X|^2}{16t^2} e^{-\frac{|X|^2}{4t}} \right)\label{trace}
\end{align}
for $A(t)>1$. We  now have to choose $A(t)$ judiciously to complete our proof. Now using Lemma \ref{hardy} we have\begin{align}
\label{Hardy}
\int_{\Rb^{n+1}_{+}} x_{n+1}^{a} w^2 \f{|X|^2}{16t^2} e^{-\frac{|X|^2}{4t}}\leq \int_{\Rb^{n+1}_{+}} x_{n+1}^a \f{n+1+a}{4t} e^{-\frac{|X|^2}{4t}} w^2+\int_{\Rb^{n+1}_{+}} x_{n+1}^a e^{-\frac{|X|^2}{4t}} |\nabla w|^2.
\end{align}
Plugging the estimate \eqref{Hardy} in \eqref{trace} yields
\begin{align}
\notag |K_1|&\leq \int t^{-\mu-1} \sigma_{s}^{-2\alpha}\int_{\Rb^n}\left(nV(x, t)+\langle x, \nabla_{x}V(x, t) \R \right)e^{-\frac{|x|^2}{4t}} w^2\\
 &\lesssim \notag   \|V\|_{1} \int t^{-\mu-1} \sigma_{s}^{-2\alpha} \left(A(t)^{1+a}\int_{\Rb^{n+1}_{+}} x_{n+1}^a e^{-\frac{|X|^2}{4t}} w^2+2A(t)^{a-1}\int_{\Rb^{n+1}_{+}} x_{n+1}^a |\nabla w|^2 e^{-\frac{|X|^2}{4t}}\right.\\
&\left.+A(t)^{a-1}\int_{\Rb^{n+1}_{+}} x_{n+1}^a \f{n+1+a}{4t} e^{-\frac{|X|^2}{4t}} w^2 \right)\notag\\
\notag & \lesssim \|V\|_{1}\left( \int A(t)^{1+a}  \sigma_{s}^{-2\alpha}\,  x_{n+1}^a G w^2+2\int A(t)^{a-1}x_{n+1}^a \sigma_{s}^{-2\alpha} |\nabla w|^2 G\right.\\
&\left.+\int A(t)^{a-1}\sigma_{s}^{-2\alpha-1} x_{n+1}^a  G\, w^2 \right).\label{k1est}
\end{align}
In the last inequality in \eqref{k1est} above, we used that $\sigma_s(t) \sim t$. 
Now we choose $A(t)>1$ in such a way that the above terms can be absorbed in the positive terms on the right hand side in \eqref{combined-est} above, i.e. in the terms  $\frac{\alpha}{N} \int x_{n+1}^{a} \sigma_{s}^{-2\alpha}(t) \f{\theta_{s}(\ld t)}{t} \ w^2  G$ and $\frac{1}{N} \int x_{n+1}^a \frac{\theta_{s}(\lambda t)}{t}  \sigma_{s}^{1-2\alpha}(t) |\nabla w|^2 G$. Therefore we require
\begin{equation}\begin{cases} A(t)^{1+a}\|V\|_{1}\lesssim  \f{\alpha}{10N} \frac{\theta_{s}(\lambda t)}{t},\\ A(t)^{a-1}\|V\|_{1}\lesssim \f{1}{10N}  \theta_{s}(\lambda t), \\ \frac{A(t)^{a-1}}{t} \|V\|_{1} \lesssim  \f{\alpha}{10N} \frac{\theta_{s}(\lambda t)}{t}.\end{cases}
\label{cases}
\end{equation}
It is easy to see that the third inequality automatically holds if the second one is satisfied since $\alpha$ is to be chosen large. Therefore, it is sufficient to choose $A(t)$ satisfying the first two inequalities. Recall that $a=1-2s,$ and if we set 
$$A(t)=\left(\f{10N \|V\|_{1}}{\theta_{s}(\lambda t)}\right)^{1/2s},$$ then the second inequality in \eqref{cases} is valid. Note that $A(t)>1$ as $\theta_{s}(t)\to 0$ as $t\to 0.$ Moreover, the above choice of $A$ will also satisfy the first inequality in \eqref{cases} if we choose large $\alpha$ such that
\begin{align}
\notag &\left(\f{10N \|V\|_{1}}{\theta_{s}(\lambda t)}\right)^{\f{2(1-s)}{2s}}\|V\|_{1}\leq \f{\alpha}{10N} \frac{\theta_{s}(\lambda t)}{t}\\
\notag\iff& \left(\f{10N \|V\|_{1}}{\theta_{s}(\lambda t)}\right)^{\f{1}{s}}\f{\theta_{s}(\lambda t)}
{10N \|V\|_{1}}\|V\|_{1}\leq \f{\alpha}{10N} \frac{\theta_{s}(\lambda t)}{t}\\
\label{qu}\iff & 10N \|V\|_{1}\leq \alpha^s t^{-s} \theta_{s}(\lambda t).
\end{align}
Finally, observe that $\theta_{s}(\lambda t)=(\lambda t)^s\left(\log \f{1}{\lambda t} \right)^{1+s}\geq (\lambda t)^s$ since $\log \f{1}{\lambda t}\geq 1$ on $[0, \f{1}{e\lambda}],$ so the inequality \eqref{qu} is ensured if we choose large $\alpha$ such that $$ \alpha^s t^{-s} (\lambda t)^s\geq 10N \|V\|_{1}.$$
Consequently since $\lambda=\alpha \delta^2,$ by choosing some arbitrary $\delta\in (0, 1),$ we conclude that the choice of $A(t)$ above satisfy the set of inequalities in \eqref{cases} provided $$\alpha^{2s}\geq (1+ N)\|V\|_{1},$$ possibly for a new universal constant $N$ which we again denote as $N.$ Before proceeding further, we make the following discursive remark. 

\begin{remark}\label{choi}
The reader should notice the critical role played by the choice of the function $\theta_s$ in obtaining the threshold $\alpha^{2s}\geq (1+ N)\|V\|_{1},$ which is a key ingredient in obtaining the desired vanishing order estimate.\end{remark}

For $K_2,$ applying integration by parts we observe
\begin{align}
|K_2|&=\bigg|\f{1}{2}\int_{\{x_{n+1}=0\}}(-\mu)t^{-\mu-1}\left(\f{t\sigma_{s}'}{\sigma_{s}}\right)^{-1/2}V(x, t) v^2+\f{1}{2}\int_{\{x_{n+1}=0\}} t^{-\mu}\f{-1}{2}\left(\f{t\sigma_{s}'}{\sigma_{s}}\right)^{-3/2} \left(\f{t\sigma_{s}'}{\sigma_{s}}\right)'  V(x, t) v^2 \notag\\
&+\f{1}{2}\int_{\{x_{n+1}=0\}} t^{-\mu}\left(\f{t\sigma_{s}'}{\sigma_{s}}\right)^{-1/2} V_{t} v^2+\f{1}{2}\int_{\{x_{n+1}=0; t=c\}} c^{-\mu}\left(\f{c\sigma_{s}'(c)}{\sigma_{s}(c)}\right)^{-1/2} V(x, c) v^2(x, c)\notag \bigg|.
\end{align}
The first and third terms on the right hand side of the above expression are bounded by $$C \|V\|_{1}\int_{\{x_{n+1}=0\}} t^{-\mu-1} \sigma_{s}^{-2\alpha} e^{-|x|^2/4t} w^2$$ since $\left(\f{t\sigma_{s}'}{\sigma_{s}}\right)\sim 1$ and $0\leq t<\f{1}{e\lambda}.$  
The second term is dominated by $\|V\|_{1} \int_{\{x_{n+1}=0\}} t^{-\mu} \bigg|-\left(\f{t\sigma_{s}'}{\sigma_{s}}\right)'\bigg| v^2,$ which in turn is bounded by  
$$C \|V\|_{1}\int_{\{x_{n+1}=0\}} t^{-\mu-1} \sigma_{s}^{-2\alpha} e^{-|x|^2/4t} w^2,$$
considering the fact that $-\left(\f{t\sigma_{s}'}{\sigma_{s}}\right)'$ is comparable to $\f{\theta_{s}(\lambda t)}{t}$ and $\theta_{s}(\lambda t)\to 0$ as $t\to 0.$ Combining the above arguments we have
\begin{align}
|K_2|&\lesssim \|V\|_{1}\int_{\{x_{n+1}=0\}} t^{-\mu-1} \sigma_{s}^{-2\alpha} e^{-|x|^2/4t} w^2+\bigg|\f{1}{2}\int_{\{x_{n+1}=0\}} c^{-\mu}\left(\f{c\sigma_{s}'(c)}{\sigma_{s}(c)}\right)^{-1/2} V(x, c) v^2(x, c) \bigg|.\label{bdry} 
\end{align}
The first term in \eqref{bdry} can be handled similarly as $K_1,$ see \eqref{k1mid} and \eqref{k1est}. For the last term in \eqref{bdry}, using trace inequality and performing similar calculations as in \eqref{k1est}, we obtain that
\begin{align}
\notag
 &\bigg|\f{1}{2}\int_{\{x_{n+1}=0; t=c\}} c^{-\mu}\left(\f{c\sigma_{s}'(c)}{\sigma_{s}(c)}\right)^{-1/2} V(x, c) v^2(x, c) \bigg|\\
\notag &\lesssim c \|V\|_{1}\int_{\{x_{n+1}=0\}} c^{-\mu-1} \sigma_{s}^{-2\alpha}(c)\,e^{-\f{|x|^2}{4c}} w^2(x, c)\\
&\lesssim \notag   \|V\|_{1}\left(c \sigma_{s}^{-2\alpha}(c) A^{1+a}\int  \,  x_{n+1}^a G(X, c) w^2(X, c)+2c A^{a-1}\sigma_{s}^{-2\alpha}(c)\int x_{n+1}^a  |\nabla w(X, c)|^2 G(X, c)\right.\\
&\left.+A^{a-1}\,c \f{n+1+a}{4c}\int \sigma_{s}^{-2\alpha}(c) x_{n+1}^a  G(X, c)\, w^2(X, c) \right)\label{name}
\end{align}
holds for any $A>1.$ If we now choose $A$ sufficiently large, say \begin{equation}\label{newa}A^{2s}\sim 100N\|V\|_{1},\end{equation}  then the term $$2c A^{a-1}\|V\|_{1}\sigma_{s}^{-2\alpha}(c)\int x_{n+1}^a  |\nabla w(X, c)|^2 G(X, c)$$ in \eqref{name} can easily be absorbed by the term $\f{c}{N} \, \sigma_{s}^{-2\alpha}(c) \int_{t=c} x_{n+1}^{a} |\nabla w|^2 G \ dX$ in \eqref{combined-est}. Corresponding to this choice of $A$ as in \eqref{newa}, we find by also  using that $c \lesssim \frac{1}{\alpha}$, the remaining terms in the last expression in \eqref{name} above can be estimated as
\begin{align*}
&\|V\|_1 \left( c \sigma_{s}^{-2\alpha}(c) A^{1+a}\int  \,  x_{n+1}^a G(X, c) w^2(X, c)+A^{a-1}\,c \f{n+1+a}{4c}\int \sigma_{s}^{-2\alpha}(c) x_{n+1}^a  G(X, c)\, w^2(X, c) \right) 
\\
&\leq N\alpha \sigma_{s}^{-2\alpha}(c) \int_{t=c} x_{n+1}^{a} w^2(X,c)\,G(X, c).\end{align*}

 Therefore, from the above discussion, the contributions from $K_1$ and $K_2$ can be absorbed appropriately by the first four terms in \eqref{combined-est} so that for large $\alpha$ satisfying $\alpha^{2s}\geq N(\|V\|_{1}+1)$ the following holds
\begin{align}
\label{finalest}
&\int \sigma_{s}^{-2 \alpha}(t) t^{-\mu} x_{n+1}^{-a} e^{-\f{|X|^2}{4t}} \left(\f{t \sigma_{s}'}{\sigma_{s}}\right)^{-\f{1}{2}}  \lvert \widetilde{\mc{H}_s} w \rvert ^2\\
&\geq \mc{I}_{1}+\mc{I}_{2}+ \mc{I}_{3}+ \mc{I}_{4}\notag\\
&\geq\f{\alpha}{N} \int x_{n+1}^{a} \sigma_{s}^{-2\alpha}(t)  \f{\theta_{s}(\lambda t)}{t}\, G\, w^2+ \frac{1}{N}  \int x_{n+1}^a \f{\theta_{s}(\lambda t)}{t} \sigma_{s}^{1-2\alpha}(t)\,\, G\,\, |\nabla w|^2\notag\\
&-N\alpha \sigma_{s}^{-2\alpha}(c) \int_{t=c} x_{n+1}^{a} w^2(X,c)\,G(X, c)+\f{c}{N}\, \sigma_{s}^{-2\alpha}(c) \int_{t=c} x_{n+1}^{a} |\nabla w|^2 G \ dX.\notag
\end{align}
Also, \eqref{large} implies that $\f{\theta_{s}(\lambda t)}{t}\gtrsim \lambda=\f{\alpha}{\delta^2},$ hence 
\begin{align}
\notag& N\int \sigma_{s}^{-2 \alpha}(t) t^{-\mu} x_{n+1}^{-a} e^{-\f{|X|^2}{4t}} \left(\f{t \sigma_{s}'}{\sigma_{s}}\right)^{-\f{1}{2}}  \lvert \widetilde{\mc{H}_s} w \rvert ^2\\
&\geq \alpha^2 \int x_{n+1}^{a} \sigma_{s}^{-2\alpha}(t) \ w^2  G +\alpha \int x_{n+1}^a   \sigma_{s}^{1-2\alpha}(t) |\nabla w|^2 G\notag\\
 &-N\alpha \sigma_{s}^{-2\alpha}(c) \int_{t=c} x_{n+1}^{a} w^2(X,c)\,G(X, c)+\f{c}{N}\, \sigma_{s}^{-2\alpha}(c) \int_{t=c} x_{n+1}^{a} |\nabla w|^2 G \ dX\label{finalest-1}
\end{align}
possibly for a new universal constant $N.$ Finally, the the result follows from \eqref{finalest-1} since $$\int_{\mb{R}^{n+1}_+ \times [c, \iy)}  \sigma_{s}^{-2 \alpha}(t) t^{-\mu} x_{n+1}^{-a} e^{-\f{|X|^2}{4t}} \left(\f{t \sigma'}{\sigma}\right)^{-\f{1}{2}}  \lvert \widetilde{\mc{H}_s} w \rvert ^2 \sim \int_{\mb{R}^{n+1}_+ \times [c, \iy)} \sigma_{s}^{ 1-2 \alpha}(t) x_{n+1}^{-a} \ \lvert \widetilde{\mc{H}_s} w \rvert ^2 \ G.$$ 
 \end{proof}
As a consequence of Theorem~\ref{carl1} and a translation   in time,  we obtain the following corollary.
\begin{corollary}
 \label{carl1-corr}
Under the conditions of Theorem~\ref{carl1}, we have for all large enough $\alpha$ satisfying \eqref{al} above
		\begin{align}\label{car1}
			& \alpha^2 \int_{\mb{R}^{n+1}_+ \times [0, \iy)} x_{n+1}^{a} (\sigma_{s}(t+c))^{-2 \alpha} \ w^2 \ G_{c}  + \alpha \int_{\mb{R}^{n+1}_+ \times [0, \iy)} x_{n+1}^{a} (\sigma_{s}(t+c))^{1-2 \alpha}\  |\nabla w|^2 \ G_{c}  \\
			\notag & \leq  M\int_{\mb{R}^{n+1}_+ \times [0, \iy)} \sigma_{s}^{ 1-2 \alpha}(t+c) x_{n+1}^{-a} \ \lvert \widetilde{\mc{H}_{s}} w \rvert ^2 \ G_c\\
   &+ \sigma_{s}^{-2 \alpha}(c) \left\{ - \f{c}{M} \int_{t=0} x_{n+1}^a \ |\nabla w(X,0)|^2 \ G(X,c) \ dX + \alpha M\int_{t=0} x_{n+1}^a \ |w(X, 0)|^2 \ G(X,c) \ dX\right\}.\notag
		\end{align}
		Here,  $G_{c}(X, t)=G(X,t+c) = \f{1}{{(t+c)^\f{n+1+a}{2}}} e^{-\f{|X|^2}{4(t+c)}}$ and $0 < c \le \f{1}{5\ld}. $   
\end{corollary}

\section{Quantitative doubling estimates}\label{s:quant} In this section we prove several quantitative doubling estimates.
\subsection{Monotonicity in time}
We start with the relevant monotonicity in time result  which is motivated by Lemma 3.1 in \cite{ABDG}. See also \cite{BVS}. This represents the nonlocal counterpart of \cite[Lemma 3.2]{AB1}.    The key novelty of our estimate is the precise quantitative dependence on the required parameters.  In combination with the quantitative Carleman estimate in Corollary \ref{carl1-corr}, this allows us to establish the precise doubling estimate for the extension problem \eqref{exprob} that eventually leads to the sharp vanishing order estimate for the nonlocal equation.

Define
\begin{equation}\label{theta}
\Theta \overset{def}{=}\frac{\int_{\mathbb Q_5^+} U(X,t)^2 x_{n+1}^adXdt }{\int_{\mathbb B_1^+} U(X,0)^2 x_{n+1}^a dX}.
\end{equation}

	\begin{lemma}\label{mont} 
		Let $U$ be a solution of \eqref{exprob}. Then there exists a constant $M = M(n,a)>2$ such that $M\operatorname{log}(M\tilde\Theta) \geq 1$, and for which the following inequality 
		\begin{align*}
Me^{\|V\|_{1}^{\f{1}{2s}}}\int_{\mathbb B_2^+}U(X,t)^2 x_{n+1}^a dX
&\geq  \int_{\mathbb B_1^+}U(X,0)^2 x_{n+1}^a dX,  
\end{align*}	
holds for $$0\leq t\leq \f{1}{M\log(M(1+\|V\|_{1})\tilde \Theta)+M^2(\|V\|_{1}^{\f{1}{2s}}+1)},$$
where
\[
\tilde \Theta= \frac{\int_{\mathbb Q_4^+} U(X,t)^2 x_{n+1}^adXdt }{\int_{\mathbb B_1^+} U(X,0)^2 x_{n+1}^a dX}.\]	\end{lemma}
\begin{proof}
Let  $f= \phi\, U,$ where $\phi \in C_0^{\infty}(\mathbb B_2)$ is a spherically symmetric cutoff such that $0\le \phi\le 1$ and $\phi \equiv1$ on $\mathbb B_{3/2}.$ Considering the symmetry of $\phi$ in $x_{n+1}$ variable and the fact that $U$ solves \eqref{exprob}, we obtain
		\begin{equation}\label{feq}
			\begin{cases}
				x_{n+1}^a f_t + \operatorname{div}(x_{n+1}^a \nabla f) = 2 x_{n+1}^a \langle\nabla U,\nabla \phi\rangle  + \operatorname{div}(x_{n+1}^a \nabla \phi) U\ \ \ \ \ \ \ \text{in} \ \mathbb Q_{4}^+,
				\\	
				f(x,0,t)= u(x,t)\phi(x,0)
				\\
				\py f(x,0, t)= Vf\ \ \ \ \ \ \ \ \ \ \ \ \ \ \ \ \ \ \ \ \ \ \ \ \ \ \ \ \ \ \ \ \ \text{in}\ Q_4.
			\end{cases}
		\end{equation}
For a fixed point $Y \in \mb{R}^{n+1}_+,$ we introduce the quantity
\begin{align*}
			H(t) = \int_{\mb{R}^{n+1}_+} x_{n+1}^a f(X,t)^2 \mc{G}(Y,X,t) dX,
		\end{align*}
		where $\mc{G}(Y,X,t) = p(y, x, t) \ p_a(x_{n+1},y_{n+1};t),$ and  $p(y,x,t)$ is the heat-kernel associated to $\left(\dd_t - \Delta_{x}\right)$ and $p_a$ is the fundamental solution of  the Bessel operator $\dd^2_{x_{n+1}} + \f{a}{x_{n+1}}\dd_{x_{n+1}}$. It is well-known that $p_a$ is given by the formula
   \begin{equation}\label{pa}
   	p_a(x_{n+1}, y_{n+1}; t) = (2t)^{- \f{1+a}{2}} e^{-\f{x_{n+1}^2 + y_{n+1}^2}{4t}} \left( \f{x_{n+1} y_{n+1}}{2t} \right)^{ \f{1-a}{2}} I_{\f{a-1}{2}}\left( \f{x_{n+1} y_{n+1}}{2t} \right),
   \end{equation}
   where $I_\nu (z)$ the modified Bessel function of the first kind defined by the series
   
   \begin{align}\label{besseries}
   	I_{\nu}(z) = \sum_{k=0}^{\infty}\frac{(z/2)^{\nu+2k} }{\Gamma (k+1) \Gamma(k+1+\nu)}, \hspace{4mm} |z| < \infty,\; |\operatorname{arg} z| < \pi.
   \end{align}
   Also, for $t>0$, $\mc{G} = \mc{G}(Y, \cdot)$ solves $\operatorname{div}(x_{n+1}^a \nabla \mc{G}) = x_{n+1}^a \partial_t \mc{G}.$ On the thin set $\{x_{n+1}=0\},$ $\mc G$ stands for $\mc G (Y,(x, 0), t).$
Now, \begin{align}\label{hprime}
			H'(t) & = 2 \int x_{n+1}^a f f_t \mc{G} + \int x_{n+1}^a f^2 \dd_t\mc{G} \\
			& = 2 \int x_{n+1}^a f f_t \mc{G} + \int f^2 \tn{div}\left(x_{n+1}^a \nabla \mc{G} \right)\notag \\
   & = 2 \int x_{n+1}^a f f_t \mc{G} - \int x_{n+1}^a \langle \nabla(f^2),  \nabla \mc{G} \rangle \notag \\
   & = 2 \int x_{n+1}^a f f_t \mc{G}+ \int \operatorname{div}( x_{n+1}^a \nabla(f^2))\mc{G}+2\int_{\{x_{n+1}=0\}} Vf^2 G\notag \\
			& = 2 \int f \mc{G} \left( x_{n+1}^a f_t + \tn{div} \left(x_{n+1}^a \cdot \nabla f \right) \right) + 2 \int x_{n+1}^a  \mc{G} |\nabla f|^2+2\int_{\{x_{n+1}=0\}} Vf^2 G\notag\\
   &=J_1+J_2+J_3,\label{ja}
		\end{align}
where $\mc G (Y,(x, 0), t)=G(y, x, t)=\frac{(4\pi)^{-\frac n2} 2^{-a}}{\Gamma((1+a)/2)} t^{-\frac{n+a+1}{2}} e^{-\frac{|x-y|^2+y_{n+1}^2}{4t}}$ for a fixed $Y\in \mathbb B_1^+,$ see (3.20) in \cite{ABDG} for more details. The following two estimates can be verified from \cite{ABDG} and \cite{BVS}. We refer the reader to equations $(3.13)$ and $(3.25)$ in \cite{ABDG}. There exists an universal constant $M>0$ such that \begin{itemize}
\item[(i)] for every $Y\in \mathbb B_1^+$ and $0<t\le 1$ we have 
\begin{equation}\label{i1}
J_1 \geq - C e^{-\frac{1}{M t}} M(1+\|V\|_{1})\int_{\mathbb Q_4^+} U^2 x_{n+1}^a dXdt.
\end{equation}
It is to be noted that (i) is a consequence of the regularity estimates as stated in Lemma \ref{reg1}. 

\item[(ii)] for every $Y\in \mathbb B_1$ and $0<t\le 1$ one has
\begin{align}\label{mk1}
|J_3| & \leq C(n,a)\|V\|_{1}\bigg(A^{1+a}\int f^2 \mc G x_{n+1}^a dX + \frac{n+a+1}{4t} A^{a-1} \int f^2 \mc G  x_{n+1}^a dX 
\\
& + A^{a-1} \int |\nabla f|^2 \mc G  x_{n+1}^a dX\bigg)
\notag
\end{align}
for all $A>1.$
\end{itemize}
Now we have to carefully choose $A.$ We let $A$ such that  $A\simeq \frac{\|V\|_{1}^{\f{1}{4s}}}{\sqrt{t}}.$ Then we obtain

\begin{align}\label{mk2}
|J_3| & \leq C(n, a)\|V\|_{1}\bigg(\|V\|_{1}^{\f{1-s}{2s}}t^{-\f{1+a}{2}}\int f^2 \mc G x_{n+1}^a dX + \frac{n+a+1}{4t} \|V\|_{1}^{-\f{1}{2}} t^{\f{1-a}{2}} \int f^2 \mc G x_{n+1}^a dX 
\\
& + \|V\|_{1}^{-\f{1}{2}} t^{\f{1-a}{2}} \int |\nabla f|^2 \mc G x_{n+1}^a dX\bigg).
\notag
\end{align}

With \eqref{i1} and \eqref{mk2}, we return to \eqref{ja} to obtain
\begin{align}\label{ab0}
H'(t)  \ge & - C e^{-\frac{1}{M t}} M(1+\|V\|_{1}) \int_{\mathbb Q_4^+} U^2 x_{n+1}^a dXdt + 2\int  |\nabla f|^2   \mc G x_{n+1}^a dX
\\
& -M\|V\|_{1}\bigg(\|V\|_{1}^{\f{1-s}{2s}}t^{-\f{1+a}{2}}\int f^2 \mc G x_{n+1}^a dX + \frac{n+a+1}{4t} \|V\|_{1}^{-\f{1}{2}} t^{\f{1-a}{2}} \int f^2 \mc G x_{n+1}^a dX \notag
\\
& + \|V\|_{1}^{-\f{1}{2}} t^{\f{1-a}{2}} \int |\nabla f|^2 \mc G x_{n+1}^a dX\bigg).
\notag
\end{align}
Now we choose $t$ such that $t M(1+\|V\|_{1}^{\f{1}{2s}})<\f{1}{M},$ then a simple calculation shows that 
\begin{equation}\label{ab1}
M \|V\|_{1}  \|V\|_{1}^{-\f{1}{2}} t^{\f{1-a}{2}} \int |\nabla f|^2 \mc G x_{n+1}^a dX \leq \int |\nabla f|^2 \mc G x_{n+1}^a dX.\end{equation}
Therefore the term $-M \|V\|_{1} \|V\|_{1}^{-\f{1}{2}} t^{\f{1-a}{2}} \int |\nabla f|^2 \mc G x_{n+1}^a dX$ in \eqref{ab0} can be absorbed in  the positive term $2 \int |\nabla f|^2 \mc G x_{n+1}^a dX$.  Using also that
\[
\|V\|_1^{\f{1-s}{2s}} \geq \|V\|_1^{-\f{1}{2}}, (\text{in view of \eqref{V1} which in particular implies that $\|V\|_1 \geq 1$)}
\]
we deduce from \eqref{ab0} that the following holds
\begin{align}
\notag H'(t)  \ge & -  e^{-\frac{1}{M t}} M(1+\|V\|_{1}) \int_{\mathbb Q_4^+} U^2 x_{n+1}^a dXdt-M\|V\|_{1} \|V\|_{1}^{\f{1-s}{2s}} t^{-\f{1+a}{2}}H(t)\\
&=- e^{-\frac{1}{M t}} M(1+\|V\|_{1}) \int_{\mathbb Q_4^+} U^2 x_{n+1}^a dXdt-M\|V\|_{1}^{\f{1}{2}+\f{1}{2s}} t^{-\f{1+a}{2}}H(t).
\label{mk3}
\end{align}
The above can also be written as 
\begin{align}
\label{mk4}
\left(e^{\big(\f{2M(1+\|V\|_{1}^{\f{1}{2}+\f{1}{2s}})}{1-a}t^{\f{1-a}{2}}\big)} H(t)\right)'\geq -M(1+\|V\|_{1}) e^{\big(\f{2M(1+\|V\|_{1}^{{\f{1}{2}+\f{1}{2s}}})}{1-a}t^{\f{1-a}{2}}\big)} e^{-\frac{1}{M t}}  \int_{\mathbb Q_4^+} U^2 x_{n+1}^a dXdt.
\end{align}
Upon integrating and using the fact that $\underset{t\to 0^+}{\lim}\ H(t) = f(Y,0)^2 = U(Y,0)^2,$ ( see (3.6) in \cite{ABDG}) we obtain

\begin{align}
\notag &e^{\big(\f{2M(1+\|V\|_{1}^{\f{1}{2}+\f{1}{2s}})}{1-a}t^{\f{1-a}{2}}\big)} H(t)\geq U(Y,0)^2 -M(1+\|V\|_{1}) \|U\|_{L^2(\mathbb Q_4^+, x_{n+1}^a)}^2 \int_{0}^{t} e^{\big(\f{2M(1+\|V\|_{1}^{\f{1}{2}+\f{1}{2s}})}{1-a}\eta^{\f{1-a}{2}}\big)} e^{-\frac{1}{M \eta}} d\eta \\
\notag &\implies e^{\big(\f{2M(1+\|V\|_{1}^{\f{1}{2}+\f{1}{2s}})}{1-a}t^{\f{1-a}{2}}\big)} H(t)\geq U(Y,0)^2 -M(1+\|V\|_{1}) \|U\|_{L^2(\mathbb Q_4^+, x_{n+1}^a)}^2 \int_{0}^{t} e^{\big(\f{2M(1+\|V\|_{1}^{\f{1}{2}+\f{1}{2s}}}{1-a}\eta^{\f{1-a}{2}}\big)} e^{-\frac{1}{M \eta}} d\eta\\
&\implies e^{\big(\f{2M(1+\|V\|_{1}^{\f{1}{2}+\f{1}{2s}})}{1-a}t^{\f{1-a}{2}}\big)} H(t)\geq U(Y,0)^2 -M(1+\|V\|_{1}) \|U\|_{L^2(\mathbb Q_4^+, x_{n+1}^a)}^2 t e^{\big(\f{2M(1+\|V\|_{1}^{\f{1}{2}+\f{1}{2s}})}{1-a}t^{\f{1-a}{2}}\big)} e^{-\frac{1}{M t}}.\label{useful}
\end{align}
Integrating now \eqref{useful} with respect to $Y\in \mathbb B_1^+$, exchanging the order of integration and using that $\int \mc G(X, Y, t)\, y_{n+1}^a dy=1,$ and finally renaming $Y$ as $X$ we obtain
\begin{align}
\label{mk5} e^{\big(\f{2M(1+\|V\|_{1}^{\f{1}{2}+\f{1}{2s}})}{1-a}t^{\f{1-a}{2}}\big)}\int_{\mathbb B_2^+}U(X,t)^2 x_{n+1}^a dX
&\geq  \int_{\mathbb B_1^+}U(X,0)^2 x_{n+1}^a dX \\
&-M(1+\|V\|_{1}) \|U\|_{L^2(\mathbb Q_4^+, x_{n+1}^a)}^2 e^{\big(\f{2M(1+\|V\|_{1}^{\f{1}{2}+\f{1}{2s}})}{1-a}t^{\f{1-a}{2}}\big)} e^{-\frac{1}{M t}}.\notag 
\end{align}
Let us choose $t$ such that $t\leq \f{1}{10M\log(M(1+\|V\|_{1})\tilde \Theta)+10M^2(\|V\|_{1}^{\f{1}{2s}}+1)}.$ Then 
\begin{align*}
&M(1+\|V\|_{1}) \|U\|_{L^2(\mathbb Q_4^+, x_{n+1}^a)}^2 e^{\big(\f{2M(1+\|V\|_{1}^{\f{1}{2}+\f{1}{2s}})}{1-a}t^{\f{1-a}{2}}\big)} e^{-\frac{1}{M t}}\\
&\lesssim M(1+\|V\|_{1}) \|U\|_{L^2(\mathbb Q_4^+, x_{n+1}^a)}^2 e^{\f{\|V\|_{1}^{1/2s}}{10M}} e^{-\frac{1}{M t}}\\
&\leq\f{M(1+\|V\|_{1})}{2M(1+\|V\|_{1})\tilde \Theta} \|U\|_{L^2(\mathbb Q_4^+, x_{n+1}^a)}^2\\
&\leq \f{1}{2}\int_{\mathbb B_1^+}U(X,0)^2 x_{n+1}^a dX.    
\end{align*}
Using this in \eqref{mk5}, we deduce that for $0\leq t\leq \f{1}{10M\log(M(1+\|V\|_{1})\tilde \Theta)+10M^2(\|V\|_{1}^{\f{1}{2s}}+1)}$, the following inequality holds 
\begin{align*}
e^{\big(\f{2M(1+\|V\|_{1}^{\f{1}{2}+\f{1}{2s}})}{1-a}t^{\f{1-a}{2}}\big)}\int_{\mathbb B_2^+}U(X,t)^2 x_{n+1}^a dX
&\geq  \int_{\mathbb B_1^+}U(X,0)^2 x_{n+1}^a dX.    
\end{align*}
Thus we have
\begin{align*}
Me^{\|V\|_{1}^{\f{1}{2s}}}\int_{\mathbb B_2^+}U(X,t)^2 x_{n+1}^a dX
&\geq  \int_{\mathbb B_1^+}U(X,0)^2 x_{n+1}^a dX,  
\end{align*}
completing the proof.
\end{proof}

\begin{corollary}\label{mont-cor}
Let $U$ be a solution of \eqref{exprob}. Then there exists a constant $M = M(n,a)>2$ such that the following inequality 
\begin{align}
\label{mk6}
M e^{\|V\|_{1}^{\f{1}{2s}}}\int_{\mathbb B^{+}_{2\rho}}U(X,t)^2 x_{n+1}^a dX
&\geq  \int_{\mathbb B^{+}_{\rho}}U(X,0)^2 x_{n+1}^a dX,  
\end{align}	
holds for $$0\leq t\leq \f{\rho^2}{10M\log(M(1+\|V\|_{1})\Theta_{\rho})+10M^2(\|V\|_{1}^{\f{1}{2s}}+1)},\,\, 0\leq \rho<1,$$
where \begin{equation}\label{theta-rho}
\Theta_{\rho} \overset{def}{=}\frac{\int_{\mathbb Q^{+}_{4}} U(X,t)^2 x_{n+1}^adXdt }{\rho^2 \int_{\mathbb B^{+}_\rho} U(X,0)^2 x_{n+1}^a dX}.
\end{equation}
\end{corollary}
\begin{proof}
The proof follows from applying Lemma~\ref{mont} to the rescaled solution  $\Tilde{U}(X, t):=U(\rho X, \rho^2 t)$ and the fact that $\int_{\mathbb Q_{4\rho}^{+}} U(X,t)^2 x_{n+1}^adXdt\leq \int_{\mathbb Q_{4}^{+}} U(X,t)^2 x_{n+1}^adXdt.$
\end{proof}

\subsection{Quantitative two-ball one-cylinder inequalities}
Using the Carleman estimate  \eqref{car1} and Lemma \ref{mont},  we now prove our main quantitative doubling estimate as in \eqref{dbthin} below by adapting some ideas in \cite{EFV_2006}. Such an estimate  finally  allows us to derive Theorem \ref{main}. The proof of Theorem \ref{db1} below is rather involved and therefore we will divide the proof into several steps for readers' convenience. 

\begin{theorem}\label{db1}
		Let $U$ be a solution of \eqref{exprob} in $\mathbb Q_5^+.$ There exists a universal large constant $M>2$ and $\rho \in (0,1)$, depending on $n$, $a$  such that for $r < 1/2,$ we have 
			\begin{align} \int_{\mathbb B_{2r}^{+}}  U(X,0)^2 x_{n+1}^a dX \leq N\int_{\mathbb B_{r}^{+}}   U(X,0)^2x_{n+1}^a\label{dbthin} dX,\end{align}
where $N=\exp\{M(\log(M(1+\|V\|_{1})\Theta_{\rho})+\|V\|_{1}^{1/2s})\}$.
	\end{theorem} 	
\begin{proof} Let us highlight the key steps in the proof. The key ingredients are  the quantitative  Carleman estimate in Corollary \ref{carl1-corr} and the quantitative  monotonicity in time result in  Lemma~\ref{mont}.
\medskip

\noindent{\textbf{Step 1:}} Let  $f= \eta(t)\phi(X) U,$ where $\phi \in C_0^{\infty}(\mathbb B_3)$ is a spherically symmetric cutoff such that $0\le \phi\le 1$ and $\phi \equiv1$ on $\mathbb B_{2}.$ Moreover, let $\eta$ be a cutoff in time such that $\eta=1$ on $[0, \f{1}{8\lambda}]$ and supported in $[0, \f{1}{4\lambda}).$ Since $U$ solves \eqref{exprob}, we see that the function $f$ solves the problem
\begin{equation}\label{feq1}
\begin{cases}
x_{n+1}^a f_t + \operatorname{div}(x_{n+1}^a \nabla f) =\phi x_{n+1}^a U\eta_{t}+ 2 x_{n+1}^a \eta\langle\nabla U,\nabla \phi\rangle  + \eta\operatorname{div}(x_{n+1}^a \nabla \phi) U\ \ \ \ \ \ \ \text{in} \ \mathbb Q_5^+,
\\	
f((x,t),0)= u(x,t)\phi(x,0)\eta(t)
\\
\py f((x,t),0)= V(x,t) f(x,t)\ \ \ \ \ \ \ \ \ \ \ \ \ \ \ \ \ \ \ \ \ \ \ \ \ \ \ \ \ \ \ \ \ \text{in}\ Q_5.
\end{cases}
\end{equation}
Since $\phi$ is symmetric in the $x_{n+1}$ variable, we have $\phi_{n+1}\equiv 0$ on $\{x_{n+1}=0\}$. Since $\phi$ is smooth, the following estimates are true, see \cite[(3.31)]{ABDG}.
\begin{equation}\label{obs1}
\begin{cases}
\operatorname{supp} (\nabla \phi) \cap \{x_{n+1}>0\}  \subset \mathbb B_3^+ \setminus \mathbb B_{2}^+
\\
|\operatorname{div}(x_{n+1}^a \nabla \phi)| \leq C x_{n+1}^a\ \mathbf 1_{\mathbb B_3^+ \setminus \mathbb B_{2}^+}.
\end{cases}
\end{equation}

\medskip

\noindent{\textbf{Step 2:}} The Carleman estimate \eqref{car1} applied to $f$ yields the following for sufficiently large $\alpha$ satisfying $\alpha\geq M(1+\|V\|_{1}^{\f{1}{2s}})$ and $0 < c \le \f{1}{5\ld}$   
\begin{align}
\notag & \alpha^2 \int_{\mb{R}^{n+1}_+ \times [0, \iy)} x_{n+1}^{a} (\sigma_{s}(t+c))^{-2 \alpha} \ f^2 \ G(X, t+c)  + \alpha \int_{\mb{R}^{n+1}_+ \times [0, \iy)} x_{n+1}^{a} (\sigma_{s}(t+c))^{1-2 \alpha}\  |\nabla f|^2 \ G(X, t+c)  \\
			\notag & \lesssim  M\int_{\mb{R}^{n+1}_+ \times [0, \iy)} \sigma_{s}^{ 1-2 \alpha}(t+c) x_{n+1}^{-a} \ \lvert \phi x_{n+1}^a U\eta_{t}+ 2 x_{n+1}^a \eta\langle\nabla U,\nabla \phi\rangle  + \eta\operatorname{div}(x_{n+1}^a \nabla \phi) U \rvert ^2 \ G(X, t+c)\\
   &+ \sigma_{s}^{-2 \alpha}(c) \left\{ - \f{c}{M} \int_{t=0} x_{n+1}^a \ |\nabla f(X,0)|^2 \ G(X,c) \ dX + \alpha M\int_{t=0} x_{n+1}^a \ |f(X, 0)|^2 \ G(X,c) \ dX\right\}\notag\\
   \notag & \lesssim  M\lambda^2\int_{\mb{R}^{n+1}_+ \times [0, \iy)} (\sigma_{s}(t+c))^{1-2 \alpha}\, G(X, t+c)\, x_{n+1}^{a} |U|^2 \mathbf 1_{[\f{1}{8\lambda}, \f{1}{4\lambda})}\\
   \notag &+M\int_{\mb{R}^{n+1}_+ \times [0, \iy)}  x_{n+1}^a \sigma_{s}^{ 1-2 \alpha}(t+c)\{|\nabla U|^2+|U|^2\}\mathbf 1_{\mathbb B_3 \setminus \mathbb B_{2}} \eta^2 G(X, t+c)\\
   &+ \sigma_{s}^{-2 \alpha}(c) \left\{ - \f{c}{M} \int_{t=0} x_{n+1}^a \ |\nabla f(X,0)|^2 \ G(X,c) \ dX + \alpha M\int_{t=0} x_{n+1}^a \ |f(X, 0)|^2 \ G(X,c) \ dX\right\}.\label{req1}
\end{align}
In the last inequality \eqref{req1} above, we also used \eqref{obs1}. It is to be noted in view of the regularity result for $U$ as in Lemma \ref{reg1}, the Carleman estimate \eqref{car1} can be applied to $f$. This  can be justified by a standard approximation argument by first considering the integrals in the region $\{x_{n+1}> \ve\}$ and then letting $\ve \to 0$.

\medskip

\noindent{\textbf{Step 3:}} Now we estimate the right hand side of the inequality \eqref{req1} for $c$ satisfying $c\leq \f{1}{8\lambda}.$ To do that let us first estimate the quantity $\sigma_{s}^{ 1-2 \alpha}(t+c)G(X, t+c)$ in $X\in \mathbb B_3^+\times [0, \f{1}{4\lambda})  \setminus \mathbb B_{2}^+\times [0, \f{1}{8\lambda}).$ Lemma~\ref{sigma} implies that $\sigma_{s}(t+c)\geq \f{t+c}{M},$ therefore, 
\begin{align}
\label{req3} \sigma_{s}^{ 1-2 \alpha}(t+c)G(X, t+c)&\leq M^{2\alpha-1}(t+c)^{1-2\alpha} \f{1}{(t+c)^{\f{n+a+1}{2}}} e^{-\f{|X|^2}{4(t+c)}}\\
&\leq M^{2\alpha-1}(t+c)^{1-2\alpha} \f{1}{(t+c)^{\f{n+a+1}{2}}}.
\label{req2}
\end{align}
Now if $t>\frac{1}{8\lambda}$ then $t+c>\frac{1}{8\lambda},$ thus \eqref{req2} implies 
\begin{align*}
\notag (\sigma_{s}(t+c))^{1-2 \alpha}G(X, t+c)\leq M^{2\alpha-1}(t+c)^{1-2\alpha-\f{n+a+1}{2}}\leq M^{2\alpha-1} (8\lambda)^{-1+2\alpha+\f{n+a+1}{2}}.
\end{align*}
In the case when $t\leq\frac{1}{8\lambda}$ we shall utilize the decay of $e^{-\f{|X|^2}{4(t+c)}}.$ Since $|X|>1,$ we have $e^{-\f{|X|^2}{4(t+c)}}\leq e^{\f{-1}{4(t+c)}}.$ Now \eqref{req3} and $t+c<\frac{1}{4}$ implies
\begin{align*}
 &\sigma_{s}^{ 1-2 \alpha}(t+c)G(X, t+c)\\
 &\lesssim  M^{2\alpha-1}(t+c)^{1-2\alpha} \f{1}{(t+c)^{\f{n+a+1}{2}}} e^{-\f{1}{4(t+c)}}\\
 &\lesssim M^{2\alpha-1}(t+c)^{1-2\alpha-\f{n+a+1}{2}} \left(2\alpha+\f{n+a+1}{2}\right)^{\big(2\alpha+\f{n+a+1}{2}\big)}(2(t+c))^{\big(2\alpha+\f{n+a+1}{2}-1\big)}\,\,(\text{using}\,\,e^x>\f{x^k}{k!}>\f{x^k}{k^k})\\
 &\lesssim  M^{2\alpha-1} \left(2\alpha+\f{n+a+1}{2}\right)^{\big(2\alpha+\f{n+a+1}{2}\big)}.
\end{align*}
Therefore, in any case we have \begin{align}\label{step3}\sigma_{s}^{ 1-2 \alpha}(t+c)G(X, t+c)\lesssim M^{2\alpha-1} \lambda^{2\alpha+\f{n+a+1}{2}}.\end{align}

\noindent{\textbf{Step 4:}} Incorporating \eqref{step3} in \eqref{req1} yields
\begin{align}
\notag & \alpha^2 \int_{\mb{R}^{n+1}_+ \times [0, \iy)} x_{n+1}^{a} (\sigma_{s}(t+c))^{-2 \alpha} \ f^2 \ G(X, t+c) + \alpha \int_{\mb{R}^{n+1}_+ \times [0, \iy)} x_{n+1}^{a} (\sigma_{s}(t+c))^{1-2 \alpha}\  |\nabla f|^2 \ G(X, t+c)  \\ 
\notag & \lesssim   M^{2\alpha+\f{n+a+1}{2}} \alpha^{2\alpha+\f{n+a+1}{2}}\int_{[0, \f{1}{4\lambda})}\int_{\mathbb B_{3}^+} x_{n+1}^{a} \{|\nabla U|^2+|U|^2\}\\
   &+ \sigma_{s}^{-2 \alpha}(c) \left\{ - \f{c}{M} \int_{t=0} x_{n+1}^a \ |\nabla f(X,0)|^2 \ G(X,c) \ dX + \alpha M\int_{t=0} x_{n+1}^a \ |f(X, 0)|^2 \ G(X,c) \ dX\right\}\notag\\
   \notag & \lesssim  M^{2\alpha+\f{n+a+1}{2}} \alpha^{2\alpha+\f{n+a+1}{2}} (1+\|V\|_{1}) \int_{\mathbb Q_{4}^{+}} x_{n+1}^{a} U^2(X, t)\ \text{(using Lemma \ref{reg1})}\\
   &+ \sigma_{s}^{-2 \alpha}(c) \left\{ - \f{c}{M} \int_{t=0} x_{n+1}^a \ |\nabla f(X,0)|^2 \ G(X,c) \ dX + \alpha M\int_{t=0} x_{n+1}^a \ |f(X, 0)|^2 \ G(X,c) \ dX\right\}.\label{step4}
\end{align}

\noindent{\textbf{Step 5:}} Since $\phi=1$ on $\mathbb{B}_{2}$ and $\eta=1$ on $[0, \f{1}{8\lambda})$ and for small enough $\rho<\f{1}{2},$ which will be chosen later, we obtain  
\begin{align}
\notag & \alpha^2 \int_{\mb{R}^{n+1}_+ \times [0, \iy)} x_{n+1}^{a} \sigma_{s}^{-2 \alpha}(t+c) \ f^2 \ G(X, t+c)\\
&\geq  \notag \alpha^2 \int_{[0, \f{1}{8\lambda})}\int_{\mathbb{B}_{2}}  x_{n+1}^{a} \sigma_{s}^{-2 \alpha}(t+c) \ U^2 \ G(X, t+c)\\
 &\geq \notag \alpha^2 \int_{[0, \f{\rho^2}{4\lambda})}\int_{\mathbb{B}_{2\rho}^{+}}  x_{n+1}^{a} \sigma_{s}^{-2 \alpha}(t+c) \ U^2 (t+c)^{-\f{n+a+1}{2}} e^{-\f{|X|^2}{4(t+c)}}\\
 &\geq \alpha^2 \int_{[0, \f{\rho^2}{4\lambda})} (t+c)^{-2\alpha} (t+c)^{-\f{n+a+1}{2}} e^{-\f{\rho^2}{(t+c)}}\int_{\mathbb{B}_{2\rho}^{+}}  x_{n+1}^{a}U^2.\label{step5-1}
\end{align}

At this point let us assume $\alpha>M(\log(M(1+\|V\|_{1})\Theta_{\rho})+\|V\|_{1}^{1/2s})$ and $0<t\leq \f{\rho^2}{4\lambda}$ and $0\leq c\leq \f{\rho^2}{8\lambda}.$ We then employ the monotonicity result, Corollary~\ref{mont-cor}, to \eqref{step5-1} to obtain
\begin{align}
\notag & \alpha^2 \int_{\mb{R}^{n+1}_+ \times [0, \iy)} x_{n+1}^{a} \sigma_{s}^{-2 \alpha}(t+c) \ f^2 \ G(X, t+c)\\
 &\geq \alpha^2 \int_{[c, c+\f{\rho^2}{4\lambda})} t^{-2\alpha} t^{-\f{n+a+1}{2}} e^{-\f{\rho^2}{t}} \f{1}{M} e^{-\|V\|_{1}^{\f{1}{2s}}}\int_{\mathbb{B}_{\rho}^{+}}  x_{n+1}^{a}U^2(X, 0)\notag\\
& \geq \alpha^2 \f{1}{M} e^{-\alpha}\int_{[\f{\rho^2}{8\lambda}, \f{\rho^2}{4\lambda})} t^{-2\alpha} t^{-\f{n+a+1}{2}} e^{-\f{\rho^2}{t}} \int_{\mathbb{B}_{\rho}^{+}}  x_{n+1}^{a}U^2(X, 0)\notag\\
&\geq \alpha^2 \f{1}{M} e^{-\alpha}\left(\f{\rho^2}{4\lambda}\right)^{-\big(2\alpha+\f{n+a+1}{2}\big)} e^{-8\lambda} \left(\frac{\rho^2}{4\lambda}\right)
\int_{\mathbb{B}_{\rho}^{+}}  x_{n+1}^{a}U^2(X, 0)\notag\\
&\geq \frac{\delta^2 4^{2\alpha+\f{n+a+1}{2}}\lambda^{2\alpha+\f{n+a+1}{2}+1}}{8 M}(e^{5/\delta^2}\rho^2)^{-2\alpha}\rho^{2-\f{n+a+1}{2}}\int_{\mathbb{B}_{\rho}^{+}}  x_{n+1}^{a}U^2(X, 0)\notag\\
&=\frac{\delta^2 4^{2\alpha+\f{n+a+1}{2}}\lambda^{2\alpha+\f{n+a+1}{2}+1}}{8 M\Theta_{\rho}}(e^{5/\delta^2}\rho^2)^{-2\alpha}\rho^{-\f{n+a+1}{2}}\int_{\mathbb{Q}_{4}^{+}}  x_{n+1}^{a}U^2(X, t).
 \label{step5}
\end{align}

\noindent{\textbf{Step 6:}} The inequalities \eqref{step4} and \eqref{step5} together imply that 
\begin{align}
\notag &\frac{\delta^2 4^{2\alpha+\f{n+a+1}{2}}\lambda^{2\alpha+\f{n+a+1}{2}+1}}{8 M\Theta_{\rho}}(e^{5/\delta^2}\rho^2)^{-2\alpha}\rho^{-\f{n+a+1}{2}}\int_{\mathbb{Q}_{4}^{+}}  x_{n+1}^{a}U^2(X, t)\\
&\notag  \lesssim  M^{2\alpha+\f{n+a+1}{2}} \alpha^{2\alpha+\f{n+a+1}{2}} (1+\|V\|_{1}) \int_{\mathbb Q_{4}^{+}} x_{n+1}^{a} U^2(X, t)\\
   &+ \sigma_{s}^{-2 \alpha}(c) \left\{ - \f{c}{M} \int_{t=0} x_{n+1}^a \ |\nabla f(X,0)|^2 \ G(X,c) \ dX + \alpha M\int_{t=0} x_{n+1}^a \ |f(X, 0)|^2 \ G(X,c) \ dX\right\}.\label{step6}
\end{align}
To absorb the first term in the right hand side into the left, we need
\begin{align}
\notag &\frac{\delta^2 4^{2\alpha+\f{n+a+1}{2}}\lambda^{2\alpha+\f{n+a+1}{2}+1}}{8 M\Theta_{\rho}}(e^{5/\delta^2}\rho^2)^{-2\alpha}\rho^{-\f{n+a+1}{2}}\geq 8M^{2\alpha+\f{n+a+1}{2}} \alpha^{2\alpha+\f{n+a+1}{2}} (1+\|V\|_{1})
\end{align}
In view of the fact that $\alpha^{2s} \sim \lambda^{2s} \geq (1+\|V\|_{1})$ the above will be guaranteed if we choose $\rho$ such that
\begin{align}
\notag& \frac{\delta^2 4^{2\alpha+\f{n+a+1}{2}}\lambda^{2\alpha+\f{n+a+1}{2}+1}}{8 M\Theta_{\rho}}(e^{5/\delta^2}\rho^2)^{-2\alpha}\rho^{-\f{n+a+1}{2}}\geq 8 M^{2\alpha+\f{n+a+1}{2}} \lambda^{2\alpha+\f{n+a+1}{2}} \lambda^{2s}\\
\notag&\impliedby  \delta^24^{2\alpha+\f{n+a+1}{2}}(e^{5/\delta^2}M\rho^2)^{-2\alpha}\geq 64 M^{\f{n+a+1}{2}+1}  \lambda^{2s}\Theta_{\rho}\\
\notag&\impliedby (e^{5/\delta^2}M\rho^2)^{-2\alpha}\geq M^{\f{n+a+1}{2}+1}  \Theta_{\rho}\,  \  \ (\text{as}  \  \ \delta^2 4^{2\alpha+\f{n+a+1}{2}}\geq 64 \lambda^{2s})\\
\notag&\impliedby (e^{5/\delta^2}M\rho^2)^{-2\alpha}\geq M^{\f{n+a+1}{2}} e^{\alpha/M} \,\,\,\, \,\, (\text{using}\,\,\, e^{\alpha/M}> M\Theta_{\rho}\,\,\text{as}\,\,\alpha>M(\log(M(1+\|V\|_{1})\Theta_{\rho}))\\
\notag&\iff (e^{5/\delta^2}M e^{{1}/{2M}}\rho^2)^{-2\alpha}\geq M^{\f{n+a+1}{2}},
\end{align}
the above inequality will be true provided $e^{5/\delta^2}M e^{{1}/{2M}}\rho^2\leq \f{1}{16}.$ Therefore, under the above mentioned condition on $\rho$ and for $\alpha=\alpha_{0}=M(\log(M(1+\|V\|_{1})\Theta_{\rho})+\|V\|_{1}^{1/2s})$, \eqref{step6} implies 
\begin{align}
\label{help}&\f{c}{M} \int x_{n+1}^a \ |\nabla f(X,0)|^2 \ G(X,c) \ dX\leq  \alpha_{0} M\int x_{n+1}^a \ |f(X, 0)|^2 \ G(X,c) \ dX\\
\implies& 2c\int  x_{n+1}^a \ |\nabla f(X,0)|^2 e^{-\f{|X|^2}{4c}}+\f{n+a+1}{2}\int x_{n+1}^a \ |f(X, 0)|^2 e^{-\f{|X|^2}{4c}} dX\notag \\
&\leq M^3 \alpha_{0} \int_{t=0} x_{n+1}^a \ |f(X, 0)|^2 e^{-\f{|X|^2}{4c}} dX.\label{stepfinal}
\end{align}
At this point \eqref{stepfinal} combined with Lemma~\ref{do} allow us to infer that
\begin{align}
\label{adb}
	\int_{\mathbb B_{2r}^+} U^2(X, 0) x_{n+1}^a dX \leq N \int_{\mathbb B_{r}^+} U^2(X, 0) x_{n+1}^a dX
\end{align}
holds for all $0\leq r<\f{1}{2},$ where $N=\exp\{M^4(\log(M(1+\|V\|_{1})\Theta_{\rho})+\|V\|_{1}^{1/2s})\}$ and 
$\Theta_{\rho}=\frac{\int_{\mathbb Q_4^{+}} U(X,t)^2 x_{n+1}^adXdt }{\rho^2 \int_{\mathbb B_{\rho}^{+}} U(X,0)^2 x_{n+1}^a dX}$ with $e^{5/\delta^2}M e^{{1}/{2M}}\rho^2\leq \f{1}{16}.$ We note that here onwards the resulting quantity $\rho<1$ will be fixed for the rest of the article. Renaming $M^4$ by $M$ again we complete the proof.
\end{proof}
 In order to obtain a dependence of the doubling constant in Theorem \ref{db1} on $\Theta$ instead of $\Theta_\rho$, we argue  by a covering argument and by using  the following two ball one-cylinder inequality. 
\begin{theorem}
\label{bc}
 Under the assumptions Theorem~\ref{db1} and for $\rho, M, N$ defined in Theorem~\ref{db1}, we have the following two-ball one-cylinder inequality for all $0< r< \rho$
 \begin{align}
 \label{ball-cylinder}
\int_{\mathbb B_{\rho}^{+}} U^2(X, 0) x_{n+1}^a dX&\leq e^{M_{1}\f{\log_{2}(\f{2\rho}{r})}{1+M\log_{2}(2\rho/r)}}\left(\int_{\mathbb B_{r}^+} U^2(X, 0) x_{n+1}^a dX\right)^{\f{1}{{1+M\log_{2}(2\rho/r)}}}\notag \\
&\left(M \int_{\mathbb Q_{4}^+} U(X,t)^2 x_{n+1}^a dXdt \right)^{\f{M\log_{2}(2\rho/r)}{{1+M\log_{2}(2\rho/r)}}},  
\end{align}
 where $M_{1}=M(\log({M}(1+\|V\|_{1})))+ M\|V\|_{1}^{1/2s}.$
\end{theorem}
\begin{proof}
Let $0< r< \rho,$ and $k\in \mathbb N $ be such that $\rho\leq 2^k r<2\rho.$ Then iterating \eqref{adb} repeatedly we get
\begin{align*}
\int_{\mathbb B_{\rho}^{+}} U^2(X, 0) x_{n+1}^a dX&\leq \int_{\mathbb B_{2^{k}r}^{+}} U^2(X, 0) x_{n+1}^a dX \leq N^{k} \int_{\mathbb B_{r}^+} U^2(X, 0) x_{n+1}^a dX\\
& \leq N^{\log_{2}(\f{2\rho}{r})} \int_{\mathbb B_{r}^{+}} U^2(X, 0) x_{n+1}^a dX.
\end{align*}
Let us rewrite $N$ as $N=e^{M(\log({M}(1+\|V\|_{1})))} e^{M\log(M \Theta_\rho)}e^{M\|V\|_{1}^{1/2s}}=e^{M_{1}+M\log(M \Theta_{\rho})},$ then
\begin{align}
\notag \int_{\mathbb B_{\rho}^+} U^2(X, 0) x_{n+1}^a dX  &\leq e^{M_{1}\log_{2}(\f{2\rho}{r})} (M \Theta_\rho)^{M\log_{2}(\f{2\rho}{r})} \int_{\mathbb B_{r}^+} U^2(X, 0) x_{n+1}^a dX\\
&=e^{M_{1}\log_{2}(\f{2\rho}{r})} \left(M \frac{\int_{\mathbb Q_4^{+}} U(X,t)^2 x_{n+1}^adXdt }{\rho^2 \int_{\mathbb B_{\rho}^{+}} U(X,0)^2 x_{n+1}^a dX}\right)^{M\log_{2}(\f{2\rho}{r})} \int_{\mathbb B_{r}^{+}} U^2(X, 0) x_{n+1}^a dX.\notag \end{align}
Since $\rho<1$ is also a fixed universal constant, we have
\begin{align*}
\int_{\mathbb B_{\rho}^{+}} U^2(X, 0) x_{n+1}^a dX&\leq e^{M_{1}\f{\log_{2}(\f{2\rho}{r})}{1+M\log_{2}(2\rho/r)}}\left(\int_{\mathbb B_{r}^+} U^2(X, 0) x_{n+1}^a dX\right)^{\f{1}{{1+M\log_{2}(2\rho/r)}}}\\
&\left(M \int_{\mathbb Q_4^{+}} U(X,t)^2 x_{n+1}^a dXdt \right)^{\f{M\log_{2}(2\rho/r)}{{1+M\log_{2}(2\rho/r)}}},  
\end{align*}
possibly for a different universal constant $M$ depending on $\rho.$  
\end{proof}
The following lemma allows us to replace $\Theta_\rho$ by $\Theta$ in the doubling inequality \eqref{dbthin}. This in turn relies on a covering argument inspired by similar ideas in \cite{EFV_2006}. However,  such an argument requires  certain  important modification in our setting   since the operator is not translation invariant in the $x_{n+1}$ direction. \begin{lemma}\label{theta}
Under the assumptions of Theorem~\ref{db1} we have a universal constant $N_{0},$ possibly depending on $\rho,$ and a constant $\tau\in (0, 1)$ such that
\begin{align}
\label{tt}\Theta_{\rho}^{\tau}\leq N^{2-\tau}_{0} e^{M_{1}(1-\tau)}  \rho^{-2\tau} \Theta,
\end{align}
where $M_1$ is defined as in Theorem~\ref{bc}.
\end{lemma} 
\begin{figure}
\begin{tikzpicture}[scale=5]

\begin{scope}[xshift=2cm] 

\draw[->] (-1.5,0) -- (1.5,0) node[below] {$ x$};
\draw[->] (0,-.2) -- (0,1.5) node[left] {$ x_{n+1}$};








\draw [->](.90,.25) -- (1.2,.40) node [right] {$\mathcal F_1$};

\draw [->](.32,.756) -- (.5,1) node [right] {$\mathcal F_2$};
\draw [->](0,.78) -- (.5,1) ;
\draw [->](.5,.7) -- (.5,1) ;
\draw  [thick] (-1,0) arc (180:0:10mm);
\draw [thick] (-1,0) arc (180:0:3.5mm);
\draw[thick] (-.35,0) arc (180:0:3.5mm);
\draw [thick] (.3,0) arc (180:0:3.5mm);

\draw (-0.5, 1.2) node[below right] {$\mathbb B_{101\rho/100}^{+}$}; 

\draw[dashed] (-1,0.09) -- (1,0.09) ;



\draw (-1.1,0.09) node[left] {\small{$x_{n+1}=c_1\rho/100$}};


\fill[gray!30,nearly transparent] (-1, 0)-- (-1,0.09) -- (1,0.09)--(1, 0) -- cycle;

\draw [dashed] (-.35,.30) circle[radius=2mm];
\draw [dashed] (.25,.57) circle[radius=2mm];
\draw  [dashed] (0,.30) circle[radius=2mm];
\draw [dashed] (.3,.3) circle[radius=2mm];
\draw [dashed] (-.72,.30) circle[radius=2mm];
\draw [dashed] (-.55,.52) circle[radius=2mm];
\draw [dashed] (-.3,.57) circle[radius=2mm];
\draw  [dashed] (0,.58) circle[radius=2mm];
\draw [dashed] (.35,.69) circle[radius=2mm];
\draw [dashed] (.60,.52) circle[radius=2mm];
\draw [dashed] (.62,.3) circle[radius=2mm];












\end{scope}

\end{tikzpicture}
\end{figure}

\begin{proof}
We employ a covering argument. Applying the inequality \eqref{ball-cylinder} to $r=\f{\rho}{2}$ we get 
\begin{align}
\label{cv1}
\int_{\mathbb B_{\rho}^{+}} U^2(X, 0) x_{n+1}^a dX&\leq e^{M_{1}(1-\beta)}\left(\int_{\mathbb B_{\rho/2}^+} U^2(X, 0) x_{n+1}^a dX\right)^{\beta}\left(M \int_{\mathbb Q_4^{+}} U(X,t)^2 x_{n+1}^a dXdt \right)^{1-\beta},  
\end{align}
where $\beta:=\frac{1}{1+M\log 4}.$ Our first step is to obtain a similar estimate as in \eqref{cv1} with  $\mathbb{B}_{101\rho/100}^{+}$  possibly for a different choice of $\beta\in (0, 1)$. We sketch the idea. A simple covering argument yields that there are two finite families of balls $\mathcal F_{1}, \mathcal F_{2}$ such that
$$\mathbb{B}_{101\rho/100}^{+}\subset \bigcup\limits_{i\in \mathcal{F}_1}\mathbb{B}_{\rho}^{+}(Z_i)\bigcup\limits_{i\in \mathcal{F}_2 }\mathbb{B}_{c_1\rho}^{+}(Y^i),$$
where
\begin{align}\notag&\mathcal{F}_1:=\{Z_i: 1\leq i\leq n_1; Z_{i}=(z_{i}, 0)\in \mathbb{B}_{\rho/2}\},\\
\notag &\mathcal{F}_2:=\{Y^i=(y^{i}, y^i_{n+1}): 1\leq i \leq n_2, Y^{i}\in \mathbb{B}_{c_{2}\rho}^{+}\},
\end{align}
and $n_1, n_2$ are dimensional constants, $c_i's$ are fixed positive constants such that the distance of the  cylinders $\mathbb{Q}_{4c_{1}\rho}^{+}(Y^i,0)$  corresponding to  the points $\{Y_i\}$ in the collection $\mathcal F_2$ are always above the   line $\{x_{n+1}=\frac{c_{1}\rho}{100}\}$ and $c_2$ is close to $1$ such that \begin{equation} \mathbb B_{c_1\rho/2} (Y^i) \subset \mathbb B_{\rho}^+.\end{equation} This in turn can be ensured by choosing $c_1$ in a way such that  $1-c_2 > \frac{c_1}{2}$.  For each $Z_{i}=(z_{i}, 0)\in \mathcal{F}_1,$ since $U(X+(z_{i}, 0), t)$ is again a solution to \eqref{exprob}, applying \eqref{cv1} to this translated solution we get
\begin{align}
\label{cv2}
\int_{\mathbb B_{\rho}^{+}(Z_{i})} U^2(X, 0) x_{n+1}^a dX&\leq e^{M_{1}(1-\beta_{1})}\left(\int_{\mathbb B_{\rho/2}^+(Z_{i})} U^2(X, 0) x_{n+1}^a dX\right)^{\beta}\left(M \int_{\mathbb Q_5^{+}} U(X,t)^2 x_{n+1}^a dXdt \right)^{1-\beta}.
\end{align}
 Also, we have replaced $\mathbb Q_4^{+}$ by $\mathbb Q_5^{+}$ to accommodate the fact that $\mathbb Q_4^{+}+Z_{i}\subset  \mathbb Q_5^{+}.$ Now summing \eqref{cv2} for all $Z_{i}\in \mathcal{F}_{1}$ and considering $\mathbb{B}_{\rho/2}^{+}(Z_{i})\subset \mathbb B_{\rho}^{+}$ we obtain 
\begin{align}
\notag&\sum_{i=1}^{n_1}\int_{\mathbb B_{\rho}^{+}(Z_{i})} U^2(X, 0) x_{n+1}^a dX\\
\notag&\leq e^{M_{1}(1-\beta)}\left(M \int_{\mathbb Q_5^{+}} U(X,t)^2 x_{n+1}^a dXdt \right)^{1-\beta} \sum_{i=1}^{n_{1}}\left(\int_{\mathbb B_{\rho/2}^+(Z_{i})} U^2(X, 0) x_{n+1}^a dX\right)^{\beta}\\
&\notag \leq e^{M_{1}(1-\beta)} n_{1} \left(M \int_{\mathbb Q_5^{+}} U(X,t)^2 x_{n+1}^a dXdt \right)^{1-\beta} \left(\int_{\mathbb B_{\rho}^+} U^2(X, 0) x_{n+1}^a dX\right)^{\beta}\\
&\notag \leq e^{M_{1}(1-\beta+\beta(1-\beta))} n_{1} \left(M \int_{\mathbb Q_5^{+}} U(X,t)^2 x_{n+1}^a dXdt \right)^{1-\beta+\beta(1-\beta)}\left(\int_{\mathbb B_{\rho/2}^+} U^2(X, 0) x_{n+1}^a dX\right)^{\beta^2}\\
&= e^{M_{1}(1-\tilde{\beta})} n_{1} \left(M \int_{\mathbb Q_5^{+}} U(X,t)^2 x_{n+1}^a dXdt \right)^{1-\tilde{\beta}}\left(\int_{\mathbb B_{\rho}^+} U^2(X, 0) x_{n+1}^a dX\right)^{\tilde{\beta}},\label{cv3}
\end{align}
where $\tilde{\beta}=\beta^2$ and we have used \eqref{cv1} in the penultimate inequality above.\\

For the balls in $\mathcal{F}_2$ we can safely apply the rescaled local estimates from \cite{EFV_2006} to obtain
\begin{align}
\label{cv4}
\int_{\mathbb B_{c_1\rho}^{+}(Y^{i})} U^2(X, 0) dX&\leq \left(\int_{\mathbb B_{c_{1}\rho/2}^+(Y^{i})} U^2(X, 0)  dX\right)^{\beta}\left(M \int_{\mathbb Q_{4c_1\rho}^{+}(Y^i,0)} U(X,t)^2 dXdt \right)^{1-\beta}
\end{align}
for the same constant $\beta$ as in \eqref{cv1}. The reader should note that in order to choose the same $\beta$ in \eqref{cv1} and \eqref{cv4}, we have to work with the largest of the two constants appearing in \eqref{adb} and its counterpart in \cite{EFV_2006}. Now by construction we observe that $x_{n+1}\sim \rho$ for all points in the domains of integration in the above inequality \eqref{cv4} and  thus consequently from \eqref{cv4} we obtain 
\begin{align}
\label{cv5}
&\int_{\mathbb B_{c_1\rho}^{+}(Y^{i})} U^2(X, 0) x_{n+1}^a dX\lesssim \left(\int_{\mathbb B_{c_{1}\rho/2}^+(Y^{i})} U^2(X, 0)  x_{n+1}^a dX\right)^{\beta}\left(M \int_{\mathbb Q_{4c_1\rho}^{+}(Y^i,0)} U(X,t)^2 x_{n+1}^a dXdt \right)^{1-\beta}\\
& \lesssim \left(\int_{\mathbb B_{c_{1}\rho/2}^+(Y^{i})} U^2(X, 0)  x_{n+1}^a dX\right)^{\beta}\left(M \int_{\mathbb Q_{5}^{+}} U(X,t)^2 x_{n+1}^a dXdt\right)^{1-\beta}, \notag\end{align}
since $\mathbb Q_{4c_1\rho}^{+}(Y^i, 0)\subset \mathbb Q_{5}^{+},$ possibly  after choosing $\rho$  smaller if required. Summing over all $Y^{i}\in \mathcal{F}_{2}$ we get
\begin{align}
\notag &\sum_{i=1}^{n_2}\int_{\mathbb B_{c_1\rho}^{+}(Y^{i})} U^2(X, 0) x_{n+1}^a dX\\
\notag&\leq \left(M \int_{\mathbb Q_{5}^{+}} U(X,t)^2 x_{n+1}^a dXdt \right)^{1-\beta} \sum_{i=1}^{n_2}  \left(\int_{\mathbb B_{c_{1}\rho/2}^+(Y^{i})} U^2(X, 0)  x_{n+1}^a dX\right)^{\beta}\\
\notag&\lesssim n_{2}  \left( \int_{\mathbb Q_{5}^{+}} U(X,t)^2 x_{n+1}^a dXdt \right)^{1-\beta}   \left(\int_{\mathbb B_{\rho}^+} U^2(X, 0)  x_{n+1}^a dX\right)^{\beta}\\
\notag&\lesssim n_{2} e^{M_{1}(1-\beta)\beta}\left(M \int_{\mathbb Q_{5}^{+}} U(X,t)^2 x_{n+1}^a dXdt \right)^{1-\beta+\beta(1-\beta)}   \left(\int_{\mathbb B_{\rho/2}^+} U^2(X, 0)  x_{n+1}^a dX\right)^{\beta^2}\, \,\,  (\text{using} \eqref{cv1})\\
&\lesssim n_{2} e^{M_{1}(1-\beta ^2)}\left(M \int_{\mathbb Q_{5}^{+}} U(X,t)^2 x_{n+1}^a dXdt \right)^{1-\beta^2}   \left(\int_{\mathbb B_{\rho}^+} U^2(X, 0)  x_{n+1}^a dX\right)^{\beta^2}. \label{cv6}
\end{align}
Combining \eqref{cv3} and \eqref{cv6} we observe that there exist constants $\tilde{\beta}\in (0, 1)$ and $N$ such that
\begin{align*}
\int_{\mathbb B_{101\rho/100}^{+}} U^2(X, 0) x_{n+1}^a dX\lesssim N e^{M_{1}(1-\tilde{\beta})}  \left(N\int_{\mathbb Q_5^{+}} U(X,t)^2 x_{n+1}^a dXdt \right)^{1-\tilde{\beta}}\left(\int_{\mathbb B_{\rho}^+} U^2(X, 0) x_{n+1}^a dX\right)^{\tilde{\beta}}.
\end{align*}
Now, iterating the above mentioned process finitely many times we obtain 
\begin{align*}
\int_{\mathbb B_{1}^{+}} U^2(X, 0) x_{n+1}^a dX\lesssim N e^{M_{1}(1-\tau)}  \left(N \int_{\mathbb Q_5^{+}} U(X,t)^2 x_{n+1}^a dXdt \right)^{1-\tau}\left(\int_{\mathbb B_{\rho}^+} U^2(X, 0) x_{n+1}^a dX\right)^{\tau},
\end{align*}
for some universal constant $N,$ possibly depending on $\rho$ and for some $\tau\in (0, 1).$ The above inequality implies that
\begin{align}
\label{thetacon}
\Theta_{\rho}^{\tau}\leq N_{0}^{2-\tau} e^{M_{1}(1-\tau)}  \rho^{-2\tau} \Theta,
\end{align}
which is the conclusion of the lemma.
\end{proof}
Using Lemma \ref{theta} in Theorem \ref{db1}, we obtain the following sharp doubling inequality where the doubling constant now depends on $\Theta_\rho$ instead of $\Theta$. More precisely, incorporating Lemma \ref{theta} in Theorem \ref{db1}, we see that  
\begin{align*} \int_{\mathbb B_{2r}^{+}}  U(X,0)^2  x_{n+1}^a dX \leq \mathcal {K}\int_{\mathbb B_{r}^{+}}   U(X,0)^2 x_{n+1}^a dX\end{align*}
holds, where
\begin{align*}
 \mathcal{K}&:=\exp\bigg \{M(\log(M(1+\|V\|_{1})N_{0}^2 e^{M_{1}{(1-\tau)}/{\tau}} \rho^{-2} \Theta^{1/\tau}))+M\|V\|_{1}^{1/2s}\bigg\}\\
&\leq \exp\bigg \{ M\log( M N_{0}^2\rho^{-2} \Theta^{1/\tau}\big)+M \log(1+\|V\|_{1})+M M_{1}(1-\tau)/\tau+M\|V\|_{1}^{1/2s}\bigg\}. 
 \end{align*}
Finally keeping in mind that $M_{1}=M(\log({M}(1+\|V\|_{1})))+ M\|V\|_{1}^{1/2s},$ and $\log(1+\|V\|_{1})\leq C_{s} \|V\|_{1}^{1/2s},$ we conclude the following
\begin{theorem}\label{db2}
		Let $U$ be a solution of \eqref{exprob} in $\mathbb Q_5^+.$ There exists a universal large constant $M>2$, depending on $n$, $a$  such that for $r < 1/2,$ we have 
			\begin{align} \int_{\mathbb B_{2r}^{+}}  U(X,0)^2  x_{n+1}^a dX \leq N\int_{\mathbb B_{r}^{+}}   U(X,0)^2 x_{n+1}^a\label{dbthin1} dX,\end{align}
where $N=\exp\{M(\log(M(1+\|V\|_{1})\Theta)+\|V\|_{1}^{1/2s})\}$.
	\end{theorem} 	
\subsection{Vanishing estimate  on cylinders}\label{dbcyl}
In this subsection, we derive our main quantitative vanishing  estimate on space time cylinders  for the extension problem \eqref{exprob} as in Theorem \ref{main-doubling-cylinder}. As mentioned in the introduction, starting from Theorem \ref{db2}, if one follows the arguments in \cite{EFV_2006} or \cite{ABDG} that leads to a space time doubling inequality from  a space like doubling inequality  as in \eqref{dbthin1} above, one obtains a doubling constant of the type $\exp\{M(\log(M(1+\|V\|_{1})\Theta)+\|V\|_{1}^{1/2s}\log \|V\|_1)\}$ which then leads to a vanishing order estimate where $\mathcal{N}$ in \eqref{df} now depends on $\|V\|_1^{1/2s}\log \|V\|_{1}$ instead of $\|V\|_1^{1/2s}$. This discrepancy in the nature of the two doubling constants can be seen in \cite[Theorem 3.5]{ABDG}. Therefore we instead argue by a compactness argument to obtain a space time vanishing order estimate  for \eqref{exprob} with the right quantitative dependence. The corresponding result is as follows.
\begin{theorem}
\label{main-doubling-cylinder}
Let $U$ be a solution of \eqref{exprob} in $\mathbb Q_5^{+}$. Then there exists  $r_1(U)>0$ such that 
\begin{align}\label{noorder1}
\int_{\mathbb Q_{r}^{+}} U(X,t)^2 x_{n+1}^a dXdt  \geq r^{N_1},
\end{align}
holds for all $r \leq r_1(U)$ where $N_1= M(\textstyle{(\int_{\mathbb B_{1}^+} U(X,0)^2 x_{n+1}^a dX)^{-1}}+\log (M\Theta)+\|V\|_{1}^{1/2s})$ where $M$ is some universal constant. \end{theorem}
\begin{proof}

The proof is divided into 2 steps.

\medskip

\emph{Step 1:}
Using Theorem \ref{db2}, we first show that 
\begin{align}
\int_{\mathbb B_r^+}U(X,0)^2 x_{n+1}^a dX\geq r^{\tilde N},\label{noorder}
\end{align}
holds for all $r \leq1/2,$ where \begin{equation}\label{N}\tilde N:=\textstyle{(\int_{\mathbb B_{1}^+} U(X,0)^2 x_{n+1}^a dX)^{-1}}+M\log (M\Theta)+M\|V\|_{1}^{1/2s},\end{equation} 
This is seen as follows. 
For notational convenience we denote $N:=\exp\{M(\log(M(1+\|V\|_{1})\Theta)+\|V\|_{1}^{1/2s})\}.$ For $r\leq1/2$ there exists $k$ such that $1\leq 2^k r< 2.$ Iterating the inequality \eqref{dbthin1} $k$ times we obtain
\begin{align}
\notag \int_{\mathbb B_{1}^+} U(X,0)^2 x_{n+1}^a dX &\leq N^k \int_{\mathbb B_r^+}U(X,0)^2 x_{n+1}^a dX\\
\notag &\leq N^{\log_{2}(2/r)}\int_{\mathbb B_r^+}U(X,0)^2 x_{n+1}^a dX\\
\notag &=\left(\f{2}{r}\right)^{\f{\log N}{\log 2}}\int_{\mathbb B_r^+}U(X,0)^2 x_{n+1}^a dX.\end{align}
Since $c\leq 2^c$, the above implies that
\begin{align}
\notag &\implies r^{\f{\log N}{\log 2}}\leq 2^{\textstyle{(\int_{\mathbb B_{1}^+} U(X,0)^2 x_{n+1}^a dXdt)^{-1}}+\f{\log N}{\log 2}}\int_{\mathbb B_r^+}U(X,0)^2 x_{n+1}^a dX\\
&\implies r^{2\f{\log N}{\log 2}+\textstyle{(\int_{\mathbb B_{1}^+} U(X,0)^2 x_{n+1}^a dX)^{-1}}}\leq (2r)^{\f{\log N}{\log 2}+\textstyle{(\int_{\mathbb B_{1}^+} U(X,0)^2 x_{n+1}^a dX)^{-1}}}\int_{\mathbb B_r^+}U(X,0)^2 x_{n+1}^a dX\notag\\ 
&\implies\int_{\mathbb B_r^+}U(X, 0)^2 x_{n+1}^a dX\geq r^{\tilde N},\label{order}
\end{align}
where $\tilde N:=\textstyle{(\int_{\mathbb B_{1}^+} U(X,0)^2 x_{n+1}^a dX)^{-1}}+M\log (M\Theta)+M\|V\|_{1}^{1/2s}.$ In the last inequality in \eqref{order} we used that since $2r \leq 1$, therefore one has  $$(2r)^{\f{\log N}{\log 2}+\textstyle{(\int_{\mathbb B_{1}^+} U(X,0)^2 x_{n+1}^a dX)^{-1}}} \leq 1.$$

\medskip

\emph{Step 2 (Conclusion):}
We show that the conclusion of the Theorem holds with $N_1=2 \tilde N$, i.e. there exists $r_1(U)$ such that for all $r \leq r_{1}(U)$, \eqref{noorder1} holds with $N_1=2\tilde N$.  If not, then there exists a sequence of $\{r_j\}$ with $r_j \to 0$ such that 
\begin{equation}\label{contr1}
\int_{\mathbb Q_{r_j}^{+}} U(X,t)^2 x_{n+1}^a dXdt  \leq r_j^{N_1}.
\end{equation}

Now define
\[
U_j(X,t) = \frac{U(r_jX, r_j^2\, t)}{\bigg(\frac{1}{ r_j^{n+1+a}}\int_{\mathbb B_{r_j/2}^+} U^2(X, 0) x_{n+1}^a dX \bigg)^{1/2}},\,\,\,\ \text{for}\ \ \  j\geq 1. 
\]
Thanks to \eqref{noorder}, each $U_j$ is well defined. We then make the following observations.
\begin{enumerate}[i)]
    \item 
By  normalization and change of variable, \begin{equation}\label{norm}\int_{\mathbb B_{1/2}^+} U_j^2 (X, 0)x_{n+1}^a dX=1.\end{equation}
\item 
Using  \eqref{noorder}, \eqref{contr1} and the fact that $N_1= 2 \tilde N$, we find that
\begin{equation}\label{equi}
\int_{\mathbb Q_1^+} U_j^2\, x_{n+1}^a dX\, dt\leq r_j^{\tilde N-2} 2^{\tilde N} \to 0\ \text{as}\ j \to \infty.\end{equation}

\end{enumerate}
Moreover, $U_j$ solves the following problem in $\mathbb Q_1^+$
\begin{equation}\label{expb1}
\begin{cases}
\operatorname{div}(x_{n+1}^a \nabla U_j) + x_{n+1}^a\partial_t U_j=0,
\\
\py U_j ((x,0),t) = r_j^{1-a} V(r_jx, r_j^2 t) U_j ((x,0),t).
\end{cases}
\end{equation}
Similar to that in \cite{ABDG}, using \eqref{equi} and  the regularity estimates in Lemma \ref{reg1}, we can assert by Arzela-Ascoli, that  up to a subsequence, $\{U_j\}$ converge  to some $U_0$ in $H^{\alpha}(\overline{\mathbb Q_{3/4}^+})$.  Moreover, $U_0$ solves    in $\mathbb Q_{3/4}^+$ \begin{equation}\label{expb2}
\begin{cases}
\operatorname{div}( x_{n+1}^a\nabla U_0) + x_{n+1}^a \partial_t U_0=0,
\\
\py U_0 ((x,0), t) = 0
\end{cases}
\end{equation}
Due to uniform convergence and \eqref{norm}, we have 
\begin{equation}\label{norm1}
\int_{\mathbb B_{1/2}^+} U_0^2(X, 0)x_{n+1}^a dX=1.\end{equation}
On the other hand, from \eqref{equi} and uniform convergence, it also follows that $U_0 \equiv 0$ in $\mathbb Q_{3/4}^+$ which contradicts \eqref{norm1}. The conclusion thus follows.

\end{proof}

\section{Proof of Theorem~\ref{main}}\label{s:proof}

\begin{proof}[Proof of Theorem \ref{main}]

 Following \cite{ABDG} and \cite{Ru1}, we prove Theorem \ref{main} using  a blowup argument. 

We  make the following claim.

\medskip

\noindent\textbf{Claim:} There exists a $\tilde{r}<\f{r_{1}(U)}{2}$ such that \begin{align}
\label{vanvanvan}
	\int_{Q_{{r}}} u^2(x, t)\, dx dt\geq \, r^{2 N_1}.  
\end{align}
for all $r\leq \tilde{r},$ where $N_1$ and $r_1(U)$ are as in Theorem \ref{main-doubling-cylinder}.  The conclusion of the Theorem would then follow by letting $\mathcal{N}= 2N_1$.

On the contrary, let us assume the \emph{Claim} is not true, then there exists a decreasing sequence of radii $\{r_j\}_{j}$ such that $r_{j}\to 0$ and 
\begin{align}
\label{contra}
 \int_{Q_{{r_{j}}}} u^2(x, t) dxdt< \, r_{j}^{2N_1}.  
\end{align}
Recall that $U$ is the solution of the extension problem corresponding to $u$ as in  \eqref{exprob}. Now define
\[
U_j(X,t) = \frac{U(r_jX, r_j^2 t)}{ \bigg(\frac{1}{r_j^{n+3+a}}\int_{\mathbb Q_{r_j}^+} U^2 x_{n+1}^a dX dt\bigg)^{1/2}},\,\,\,\ \text{for}\ \ \  j\geq 1. 
\]
Thanks to \eqref{noorder1}, each $U_j$ is well defined. We then make the following observations.
\begin{enumerate}[i)]
    \item 
By the normalization, $\int_{\mathbb Q_1^+} U_j^2 x_{n+1}^a dX dt=1.$
\item 
Using the ``qualitative'' space time  doubling inequality in  \cite[Theorem 3.5 (iii)]{ABDG}  and change of variable we get \begin{equation}\label{nondeg}
\int_{\mathbb Q_{1/2}^+} U_j^2 x_{n+1}^a dX dt =\f{\int_{\mathbb Q_{{r_j}/{2}}^+} U^2 x_{n+1}^a dX dt}{\int_{\mathbb Q_{r_j}^+} U^2 x_{n+1}^a dX dt}\geq c>0.
\end{equation}
\end{enumerate}
Moreover, $U_j$ solves the following problem in $\mathbb Q_1^+$
\begin{equation}\label{expb1}
\begin{cases}
\operatorname{div}(x_{n+1}^a \nabla U_j) + x_{n+1}^a\partial_t U_j=0,
\\
\py U_j ((x,0),t) = r_j^{1-a} V(r_jx, r_j^2 t) U_j ((x,0),t).
\end{cases}
\end{equation}
Similar to that in \cite{ABDG}, using the regularity estimates in Lemma \ref{reg1} and Arzela-Ascoli, up to a subsequence we have that $\{U_j\}$ converge  to some $U_0$ in $H^{\alpha}(\overline{\mathbb Q_{3/4}^+})$.  Moreover, $U_0$ solves    in $\mathbb Q_{3/4}^+$ \begin{equation}\label{expb2}
\begin{cases}
\operatorname{div}( x_{n+1}^a\nabla U_0) + x_{n+1}^a \partial_t U_0=0,
\\
\py U_0 ((x,0), t) = 0
\end{cases}
\end{equation}

A change of variable and \eqref{noorder1} gives
\begin{align*}
\int_{Q_1}  U_j((x,0),t)^2 dxdt&\leq r_j^{a+1-N_1} \int_{Q_{r_j}} U((x,0),t)^2 dx dt\\
&= r_j^{a+1-N_1} \int_{Q_{r_j}} u(x,t)^2 dx dt\\
&\stackrel{\eqref{contra}}{<} r_j^{a+1+N_1}.
\end{align*}
Then by taking limit $j\to \infty$ we get $\int_{Q_1}  U_j((x,0),t)^2 dxdt\to 0$ as $j\to \infty.$ Since $U_j \to U_0$ uniformly in $\mathbb Q_{1/2}^+$ up to $\{x_{n+1}=0\}$, we must have $U_0 \equiv 0$ in $\mathbb Q_{1/2} \cap \{x_{n+1}=0\}$. Again, since $U_0$ solves the problem \eqref{expb2}, applying weak unique continuation result as in Proposition \ref{wucp}, we conclude that $U_0 \equiv 0$ in $\mathbb Q_{1/2}^+$. But this is a contradiction, as $U_j\to U_{0}$ uniformly in $\mathbb Q_{1/2}^+$, from the non-degeneracy condition \eqref{nondeg} we have $\int_{\mathbb Q_{1/2}^+} U_0^2 x_{n+1}^a dX dt>0.$ Hence the \emph{Claim} is true. This completes the proof by observing that we are working with the backward version of the problem as in \eqref{exprob}.

\end{proof}


\begin{thebibliography}{10}



\bibitem{AryaBan}V. Arya \& A. Banerjee. \emph{Quantitative uniqueness for fractional heat type operators,}  Calc. Var. Partial Differential Equations \textbf{62}~ (2023), no. 7, Paper No. 195, 47 pp.

\bibitem{AB1}
V. Arya \& A. Banerjee. \emph{Space like quantitative uniqueness for parabolic operators,}  J. Math. Pures Appl. (9) \textbf{177}~ (2023), 214-259.


\bibitem{ABDG}
V. Arya, A. Banerjee, D. Danielli \& N. Garofalo. \emph{Space-like strong unique continuation for some fractional parabolic equations,}
{ J. Funct. Anal.}, \textbf{284} (2023), no. 1, Paper No. 109723, pp. 38.

\bibitem{AT}
A. Audrito, S. Terracini.  \emph{On the nodal set of solutions to a class of nonlocal parabolic reaction-diffusion equations,}  arXiv:1807.10135, to appear in Memoirs of AMS.


\bibitem{Bk}
L. Bakri. \emph{Quantitative uniqueness for Schr\"odinger operator,} {Indiana Univ. Math. J.,} \textbf{61}~(2012), no. 4, 1565--1580.



\bibitem{BG}
A. Banerjee \& N. Garofalo.
\emph{Monotonicity of generalized frequencies and the strong unique
  continuation property for fractional parabolic equations,} { Adv. Math.}, \textbf{336} (2018), 149--241.

\bibitem{BG1}
A. Banerjee \& N. Garofalo.
\emph{On the space-like analyticity in the extension problem for nonlocal
  parabolic equations}, { Proc. Amer. Math. Soc. }\textbf{151}~ (2023), no. 3, 1235--1246.
  
  \bibitem{BG2}
A.  Banerjee \& N. Garofalo. \emph{Quantitative uniqueness for elliptic equations at the boundary of 
$C^{1, Dini}$ domains,}  {J. Differential Equations,} \textbf{261} (2016), no. 12, 6718--6757.	  
  
 
\bibitem{BG24} A. Banerjee \& A. Ghosh. \emph{Decay at infinity for solutions to some fractional parabolic equations.} Proceedings of the Royal Society of Edinburgh: Section A Mathematics. Published online 2024:1-37. doi:10.1017/prm.2024.9		
		
		
		
		
\bibitem{BVS}A. Banerjee, V. P. Krishnan \& S. Senapati.  \emph{The Calderón problem for space-time fractional parabolic operators with variable coefficients,} arXiv:2205.12509.



\bibitem{BS1}
A. Banerjee \& S. Senapati. \emph{Extension problem for the fractional parabolic Lam\'e operator and unique continuation}, arXiv:2208.11598.

\bibitem{BL}
	K. Bellova \& F. Lin. \emph{Nodal sets of Steklov eigenfunctions,} {Calc. Var. Partial Differential Equations,} \textbf{54}~ (2015), no. 2, 2239--2268.
	
\bibitem{BS}
A. Biswas \& P. Stinga.
\emph{Regularity estimates for nonlocal space-time master equations in
  bounded domains,}
\newblock { J. Evol. Equ.}, \textbf{21} (2021), no. 1, 503--565.

 \bibitem{BK}
		J. Bourgain \& C. Kenig. \emph{On localization in the continuous Anderson-Bernoulli model in higher dimension,} {Invent. Math.,} \textbf{161}~(2005), 389--426.



\bibitem{CS}
L. Caffarelli \& Luis Silvestre.
\newblock\emph{An extension problem related to the fractional {L}aplacian,}
{Comm. Partial Differential Equations}, \textbf{32}~(2007), no. 7--9,1245--1260.


		\bibitem{CK}
		G. Camliyurt \& I. Kukavica. \emph{Quantitative unique continuation for a parabolic equation},  {Indiana Univ. Math. J.,} \textbf{67}~ (2018), no. 2, 657--678.



\bibitem{DF1}
	H. Donnelly \& C. Fefferman. \emph{Nodal sets of eigenfunctions on Riemannian manifolds}, {Invent. Math.,} \textbf{93}~ (1988), 161--183.
	
	
	\bibitem{DF2}
	H. Donnelly \& C. Fefferman. \emph{Nodal sets of eigenfunctions: Riemannian manifolds with boundary,} {Analysis, et cetera,} Academic Press, Boston, MA, (1990), 251--262.	

\bibitem{EF_2003}
L.~Escauriaza \& F.~J. Fern\'{a}ndez.
\newblock\emph{Unique continuation for parabolic operators,}
\newblock { Ark. Mat.}, \textbf{41}~(2003), no. 1, 35--60.


\bibitem{EFV_2006}
L.~Escauriaza, F.~J. Fern\'{a}ndez \& S.~Vessella.
\newblock \emph{Doubling properties of caloric functions,}
\newblock { Appl. Anal.}, \textbf{85} (2006), no. 1--3, 205--223.

\bibitem{EV}
	L. Escauriaza \& S. Vessella. \newblock {\em Optimal three cylinder inequalities for solutions to parabolic equations with Lipschitz leading coefficients}, { Inverse Problems: Theory and Applications,} Cortona/Pisa, 2002, in: Contemp. Math., vol. 333, Amer. Math. Soc., Providence, RI, 2003, pp. 79--87.


\bibitem{FF}
M. Fall \& V. Felli.
\newblock{\em Unique continuation property and local asymptotics of solutions to
  fractional elliptic equations,}
\newblock { Comm. Partial Differential Equations}, \textbf{39} (2014), no. 2, 354--397.


\bibitem{FPS}
V. Felli, A. Primo \& G. Siclari, \emph{On fractional parabolic equations with hardy-type potentials}, arXiv:2212.03744v2.


\bibitem{Gcm}
N. Garofalo.
\newblock{\em Two classical properties of the {B}essel quotient {$I_{\nu+1}/I_\nu$}
  and their implications in pde's,} Advances in harmonic analysis and partial differential
  equations, Contemp. Math., vol. 748, Amer. Math.
  Soc., Providence, RI, 2020, pp. 57--97.

\bibitem{GL1}
N. Garofalo \& F. Lin.
\newblock {\em Monotonicity properties of variational integrals, {$A_p$} weights and
  unique continuation,} \newblock {Indiana Univ. Math. J.}, \textbf{35} (1986), no. 2, 245--268.

\bibitem{GL2}
N. Garofalo \& F. Lin.
\newblock {\em Unique continuation for elliptic operators: a geometric-variational
  approach,} \newblock {Comm. Pure Appl. Math.}, \textbf{40}~(1987), no. 3, 347--366.


\bibitem{GSU}
T. Ghosh, M.  Salo \& G. Uhlmann.
\newblock{\em The {C}alder\'{o}n problem for the fractional {S}chr\"{o}dinger
  equation,} \newblock {Anal. PDE}, \textbf{13} (2020), no. 2, 455--475.

\bibitem{GR}
I.~S. Gradshteyn \& I.~M. Ryzhik.
\newblock {\em Table of integrals, series, and products,} 7th ed., Elsevier/Academic Press, Amsterdam, 2007.




\bibitem{Jr1}
B. F. Jones. \emph{Lipschitz spaces and the heat equation,} {J. Math. Mech.,} \textbf{18}~(1968/69), 379--409. 

\bibitem{Jr}
B. F. Jones. \emph{A fundamental solution for the heat equation which is supported in a strip}, {J. Math. Anal. Appl.,} \textbf{60}~(1977), 314--324. 


\bibitem{Ku}
I. Kukavica, \emph{Quantitative  uniqueness for second order  elliptic operators}, {Duke Math. J.,} \textbf{91}~(1998), 225--240.

\bibitem{Ku2}
I. Kukavica, \emph{Quantitative, uniqueness, and vortex degree estimates for solutions of the Ginzburg-Landau equation}, {Electron. J. Differential Equations}, (2000), no. 61,  pp. 15.

\bibitem{KL}
		I. Kukavica \&  Q. Le, \emph{On quantitative uniqueness for parabolic equations}, J. Differential Equations, \textbf{341}~ (2022), 438--480.
		
\bibitem{LLR}
R. Lai, Y. Lin \& A. R\"{u}land.
\newblock {\em The {C}alder\'{o}n problem for a space-time fractional parabolic
  equation,}
\newblock {SIAM J. Math. Anal.}, \textbf{52} (2020), no. 3, 2655--2688.

\bibitem{Le}
N.~N. Lebedev.
\newblock {\em Special functions and their applications,} Revised edition, translated from the Russian and edited by Richard A. Silverman. Unabridged and corrected republication, Dover Publications, Inc., New York, 1972.



\bibitem{Li}
G. Lieberman.
\newblock {\em Second order parabolic differential equations,}
\newblock{World Scientific Publishing Co., Inc.,} River Edge, NJ, 1996.

\bibitem{LM}
J.-L. Lions \& E.~Magenes. \emph{Non-homogeneous boundary value problems and applications. Vol. I.,} Translated from the French by P. Kenneth. Die Grundlehren der mathematischen Wissenschaften, Band 181, Springer-Verlag, New York-Heidelberg, 1972. 
  
  
 \bibitem{L1}
A. Logunov. {\em Nodal sets of Laplace eigenfunctions: proof of Nadirashvili's conjecture and of
the lower bound in Yau's conjecture,} {Ann. of Math.,} \textbf{187}~(2018), 241--262.


\bibitem{L2}
A. Logunov. {\em Nodal sets of Laplace eigenfunctions: polynomial upper estimates of the Hausdorff measure,} {Ann. of Math.,} \textbf{187}~(2018), 221--239. 

\bibitem{LM}
A. Logunov \& E. Malinnikova. {\em Nodal sets of Laplace eigenfunctions: estimates of the
Hausdorff measure in dimension two and three,} {50 years with Hardy spaces,}  Oper. Theory Adv. Appl., 333--344. 

  \bibitem{Me}
V. Meshov. {\em On the possible rate of decrease at infinity of the solutions of second-order partial differential equations,} {Math. USSR-Sb.,} \textbf{72}~ (1992), no. 2, 343--361.  
  


\bibitem{NS}
K.~Nystr\"{o}m \& O.~Sande.
\newblock{\em Extension properties and boundary estimates for a fractional heat
  operator,}
\newblock { Nonlinear Anal.}, \textbf{140} (2016), 29--37.

\bibitem{Po}
C. Poon.
\emph{Unique continuation for parabolic equations,}
\newblock {Comm. Partial Differential Equations}, \textbf{21} (1996), no. 3--4, 521--539.

\bibitem{Ru}
A.~R\"uland.
\emph {On some rigidity properties in PDEs,} 2013,
\newblock {Dissertation, Rheinischen Friedrich-Wilhelms-Universit\"at Bonn}.

\bibitem{Ru1}
A.~ R\"uland.
\newblock {\em On quantitative unique continuation properties of fractional Schr\"{o}dinger equations: doubling, vanishing order and nodal domain estimates,} {Trans. Amer. Math. Soc.,} \textbf{369} (2017), no. 4, 2311--2362.

\bibitem{Ru2}
A. R\"{u}land.
\emph{Unique continuation for fractional {S}chr\"{o}dinger equations with
  rough potentials,}
\newblock {Comm. Partial Differential Equations}, \textbf{40} (2015), no. 1, 77--114.

\bibitem{RS}
A. R\"{u}land \& Mikko Salo.
\emph{Quantitative approximation properties for the fractional heat
  equation,}
\newblock {Math. Control Relat. Fields}, \textbf{10} (2020), no. 1, 1--26.


\bibitem{RW}
A. R\"uland \& J.N. Wang. {\em On the fractional Landis conjecture,} {J. Funct. Anal.,} \textbf{277} (2019), no. 9, 3236--3270.

\bibitem{Samko}
S. Samko.
\newblock {\em Hypersingular integrals and their applications}, volume~5 of
  {Analytical Methods and Special Functions,}
\newblock Taylor \& Francis Group, London, 2002.

\bibitem{ST}
P. Stinga \& J. Torrea.
\newblock{\em Regularity theory and extension problem for fractional nonlocal
  parabolic equations and the master equation,}
\newblock{SIAM J. Math. Anal.,} \textbf{49} (2017),  no. 5, 3893--3924.




\bibitem{Ve}
	S. Vessella. {\em Carleman estimates, optimal three cylinder inequality, and unique continuation properties for solutions to parabolic equations,} {Comm. Partial Differential Equations,} \textbf{28}~ (2003),  637--676.
	
\bibitem{Yau}
S. T. Yau. \emph{Seminar on differential geometry}, \textbf{102}, {Princeton University Press,} 1982. 

\bibitem{Yu}
H. Yu. \emph{Unique continuation for fractional orders of elliptic equations,}
{Ann. PDE.,} \textbf{3} (2017), no. 2, Paper No. 16, pp. 21.


\bibitem{Zhu1}
J. Zhu. \emph{Quantitative uniqueness for elliptic equations}, {Amer. J. Math.,} \textbf{138}~(2016), 733--762.	

\bibitem{Zhu2}
J. Zhu. \emph{Quantitative uniqueness of solutions to parabolic equations}, {J. Funct. Anal.,} \textbf{275}~ (2018), no. 9, 2373--2403.

\bibitem{Zhu0}
J. Zhu. \emph{Doubling property and vanishing order of Steklov eigenfunctions}, {Comm. Partial Differential Equations,} \textbf{40}~ (2015), no. 8, 1498--1520.	


\end{thebibliography}
\end{document}